\documentclass[11pt]{article}
\usepackage{CJK}
\usepackage{amsmath,amsfonts,mathrsfs,amssymb,amsthm,bm}
\usepackage{indentfirst}
\numberwithin{equation}{section}
\usepackage{color}

\usepackage{hyperref}
\hypersetup{hypertex=true,
	colorlinks=true,
	linkcolor=blue,
	anchorcolor=blue,
	citecolor=blue}

\newtheorem{theorem}{Theorem}[section]

\newtheorem{lemma}{Lemma}[section]

\numberwithin{equation}{section}

\newcommand{\HT}{\CJKfamily{hei}}

\def\lb{\label}

\def\be{\begin{equation}}
\def\ee{\end{equation}}
\def\bea{\begin{eqnarray}}
\def\eea{\end{eqnarray}}
\def\bes{\begin{eqnarray*}}
\def\ees{\end{eqnarray*}}

\def\y{\begin{eqnarray*}}
\def\ey{\end{eqnarray*}}

\begin{document}
\title{\HT {An improved Trudinger-Moser inequality involving $N$--Finsler--Laplacian and $L^p$ norm  }}
\author{\small {Yanjun Liu$^{a, b}$  }\\
\small $^{a}$School of Mathematical Sciences, Nankai University, Tianjin 300071, P. R. China\\
\small $^{b}$Albert-Ludwigs-Universit\"at Freiburg, Mathematisches Institut, Eckerstr. 1, \\
   \small                D-79104 Freiburg, Germany}

\date{}

\maketitle \footnote[0] {Email: liuyj@mail.nankai.edu.cn(Y. Liu).}\\
\noindent{\small
{\bf Abstract:} Suppose $F: \mathbb{R}^{N} \rightarrow [0, +\infty)$ be a  convex function of class $C^{2}(\mathbb{R}^{N} \backslash \{0\})$ which is even and positively homogeneous of degree 1.  We denote $\gamma_1=\inf\limits_{u\in W^{1, N}_{0}(\Omega)\backslash \{0\}}\frac{\int_{\Omega}F^{N}(\nabla u)dx}{\| u\|_p^N}.$
and define the norm
$\|u\|_{N,F,\gamma, p}=\bigg(\int_{\Omega}F^{N}(\nabla u)dx-\gamma\| u\|_p^N\bigg)^{\frac{1}{N}}.$
Let $\Omega\subset \mathbb{R}^{N}(N\geq 2)$ be a smooth bounded domain.  Then for  $p> 1$ and $0\leq \gamma <\gamma_1$, we have
$$ \sup_{u\in W^{1, N}_{0}(\Omega), \|u\|_{N,F,\gamma, p}\leq 1}\int_{\Omega}e^{\lambda |u|^{\frac{N}{N-1}}}dx<+\infty,  $$
where $0<\lambda \leq \lambda_{N}=N^{\frac{N}{N-1}} \kappa_{N}^{\frac{1}{N-1}}$ and $\kappa_{N}$ is the volume of a unit Wulff ball. Moreover,  by using blow-up analysis and capacity technique, we prove that the
supremum  can be attained for any $0 \leq\gamma <\gamma_1$.
}

\section{Introduction and main results}\label{section 1}
Suppose $\Omega \subset \mathbb{R}^{N}(N \geq 2)$ be a bounded smooth domain. When $1 < p < N$, the Sobolev embedding theorem implies that
 $W_{0}^{1,p}(\Omega)\hookrightarrow L^{q}(\Omega)$ is continuous for $1 \leq q \leq \frac{Np}{N-p}$. In particular, $W_{0}^{1,N}(\Omega)\hookrightarrow L^q(\Omega)$ for $1 \leq q < \infty$, but the embedding  $W_{0}^{1,N}(\Omega)\not \hookrightarrow L^\infty(\Omega)$. A counterexample is given by the function $u(x) = (- \ln| \ln |x||)_{+}$ as $\Omega$ is the unit ball. It was proposed independently by Yudovich \cite{Yudovich}, Pohozaev \cite{Pohozaev},  Peetre \cite{Peetre} and Trudinger \cite{N.S. Trudinger} that $W^{1, N}_{0}(\Omega)$ is embedded in the Orlicz space $L_{\varphi_{\alpha}}(\Omega)$ determined by the Young function $\varphi_{\alpha}(t)=e^{\alpha|t|^{\frac{N}{N-1}}}-1$ for some positive number $\alpha$, it was sharpened by Moser \cite{J. Moser} who found the best exponent and proved the following result:\\
\textbf{Theorem A} There exists a sharp constant $\alpha_{N}:=N^{\frac{N}{N-1}}\omega_{N}^{\frac{1}{N-1}}$ such that
\begin{equation}\label{1.1}
 \sup_{u\in W_0^{1, N}(\Omega), \|\nabla u\|_{N}\leq1}\int_{\Omega}e^{\alpha|u|^{\frac{N}{N-1}}}dx<+\infty,  \forall \alpha \leq  \alpha_{N},
\end{equation}
where $\omega_{N}$ is the volume of
unit ball in $\mathbb{R}^{N}$.
Moreover, the supremum in ($\ref{1.1}$) is $+\infty$ if $\alpha > \alpha_N$. Related inequalities for unbounded domains were proposed by D. M. Cao \cite{D.M. Cao} in dimension two and J. M. do $\acute{O}$ \cite{J. M. do}, Adachi-Tanaka \cite{AT} in high dimension,  however they just considered the subcritical Trudinger-Moser inequality. Ruf \cite{Ruf} (for the case $N=2$), Li and Ruf  \cite{LB}  (for the general case $N \geq 2$) obtained the Trudinger-Moser inequality in the critical case by replacing the Dirichlet norm with the standard Sobolev norm in $W^{1, N}(\mathbb{R}^{N})$. Subsequntly,   Masmoudi and Sani \cite{Masmoudi}  derived Trudinger-Moser inequalities with the exact growth condition in $\mathbb{R}^n$, These inequality  plays an important role in geometric analysis and partial differential equations, we refer  to \cite{CT03,de Figueiredo, Adimurthi,ddR, LL14, do, Masmoudi} and references therein. In \cite{LLZ},   Lam, Lu and Zhang  provide a precise relationship between  subcritical and
critical  Trudinger-Moser inequality. The similar result in Lorentz-Sobolev norms was also proved
by Lu and Tang \cite{LT}. Trudinger-Moser inequality for first order derivatives was extended to high order derivatives by D. Adams \cite{D. Adams} for
bounded domains when  dimension $N\geq2$. B. Ruf and F. Sani \cite{B. Ruf} studied the Adams type inequality with
higher derivatives of even orders for unbounded domains in $\mathbb{R}^{N}$. In \cite{LL}, Lam and Lu applied a rearrangement-free argument to prove sharp Adams' inequality in general case.

One important problem on Trudinger-Moser inequalities is whether or not extremal functions
exist. Existence of extremal functions for the Trudinger-Moser inequality was first obtained
by Carleson-Chang \cite{CC} when $\Omega$ is the unit ball, by M. Struwe \cite{S} when $\Omega$  is close to the ball in
the sense of measure, then by M. Flucher \cite{F} and K. Lin \cite{L} when $\Omega$  is a general bounded smooth
domain. Recently based on the work by Malchiodi and Martinazzi in \cite{Malchiodi-Martinazzi}, Mancini and Martinazzi \cite{Mancini-Martinazzi} reproved the Carleson and Chang's result by using a new method based on the Dirichlet energy, also allowing
for perturbations of the functional. In the entire Euclidean space, existence of extremal functions was proved by Ruf \cite{Ruf} (for the case $N=2$) and Li and Ruf  \cite{LB}  (for the general case $N \geq 2$). For extremal functions of singular version, Csat$\acute{o}$ and Roy \cite{CR} proved
that extremal functions exist in bounded domain of  2 dimension. Li and Yang \cite{LY} proved that extremal functions  exist in the entire Euclidean space.

 Moreover, there are some extensions of the Trudinger-Moser inequality. Let $\alpha_1(\Omega)$ be the first eigenvalue of the Laplacian,  Adimurthi and O. Druet \cite{AD} proved that
\begin{equation*}
 \sup_{u\in W_0^{1, 2}(\Omega),  \int_{\Omega}|\nabla u|^{2}dx\leq1}\int_{\Omega}e^{4 \pi u^2(1+\alpha \|u\|_2^2)}dx<+\infty
\end{equation*}
for  $0\leq \alpha <\alpha_1(\Omega)$, the  supremum is infinity for any $\alpha \geq\alpha_1(\Omega)$. This result was generalized by Yang \cite{Yang1,Yang2} to the cases
of high dimension and a compact Riemannian surface.    Lu-Yang \cite{LuYang} and J. Zhu \cite{J.Zhu} considered the case involving
the $L^p$ norm for any $p > 1$.   For existence of
extremal functions  of Adimurthi-Druet type
inequalities, they proved in \cite{Yang2, LuYang} that supremums $(N = 2)$  are attained for sufficiently small $\alpha\geq 0$, and that the supremum $(N\geq3)$\cite{Yang1} is attained
for all $\alpha, 0 \leq \alpha <\alpha_1(\Omega).$  Subsequently,  J.M. do $\acute{O}$ and M. de Souza   generalized the similar result 
in whole Euclidean space \cite{J.M.do1} and high dimension case \cite{J.M.do3}, and the existence of extremal functions was also obtained. A stronger version was established by Tintarev \cite{Tintarev}, namely,
\begin{equation}\label{AAA}
 \sup_{u\in W_0^{1, 2}(\Omega), \int_{\Omega}|\nabla u|^{2}dx-\alpha\| u\|_2^2 \leq1}\int_{\Omega}e^{4 \pi u^2}dx<+\infty, ~~~~0\leq \alpha <\alpha_1(\Omega),
\end{equation}
Yang \cite{Yang3} obtained extremal functions for \eqref{AAA}, which was also extended to singular version (see \cite{YZ}).  In \cite{Nguyen}, the author extends the result of Tintarev to the higher dimension as the following result:\\
\textbf{Theorem B.} Let $\Omega\subset \mathbb{R}^{N}(N\geq 2)$ be a smooth bounded domain and define
 $$\alpha(\Omega)=\inf_{u\in W^{1, N}_{0}(\Omega), u\not\equiv0}\frac{\int_{\Omega}|\nabla u|^Ndx}{\| u\|_N^N}.$$
 Then for any $0\leq \alpha <\alpha(\Omega),$
 \be  \sup_{u\in W^{1, N}_{0}(\Omega), \int_{\Omega}|\nabla u|^Ndx-\alpha\| u\|_N^N\leq 1}\int_{\Omega}e^{\alpha_N |u|^{\frac{N}{N-1}}}dx<+\infty,  \lb{1.3}\ee
where $\alpha_{N}:=N^{\frac{N}{N-1}}\omega_{N}^{\frac{1}{N-1}}$ and $\omega_{N}$ is the volume of 
unit ball in $\mathbb{R}^{N}$.

Another interesting research is that Trudinger-Moser inequality has been generalized to the case of anisotropic norm.  In this paper, denote that $F \in C^{ 2}(\mathbb{R}^{N} \backslash {0})$ is a positive, convex and homogeneous function, $F_{\xi_{i}} = \frac{\partial F}{\partial\xi_{i}}$ and its polar $F^{o}(x)$ represents a Finsler metric on $\mathbb{R}^{N}$. We will replace the isotropic Dirichlet norm $\|u\|_{W_0^{1, N}(\Omega)}=(\int_\Omega |\nabla u|^Ndx)^{\frac{1}{N}}$ by the anisotropic Dirichlet norm $(\int_\Omega F^N(\nabla u)dx)^{\frac{1}{N}}$ in $W_0^{1, N}(\Omega)$. In \cite{WX}, Wang and Xia proved the following result:\\
\textbf{Theorem C (Anisotropic Trudinger-Moser Inequality).} Let $\Omega\subset \mathbb{R}^{N}(N\geq 2)$ be a smooth bounded domain. Let $u\in W_0^{1, N}(\Omega)$ and $(\int_{\Omega}F^{N}(\nabla u)dx)\leq 1$. Then
   \be \sup_{u\in W_0^{1, N}(\Omega), \int_{\Omega}F^N(\nabla u)dx\leq1}\int_{\Omega}e^{\lambda_N u^{\frac{N}{N-1}}}dx  < +\infty, \lb{1.4}\ee
where $\lambda_{N}=N^{\frac{N}{N-1}} \kappa_{N}^{\frac{1}{N-1}}$ and $\kappa_{N}=\{x \in \mathbb{R}^{N} : F^{o}(x)\leq 1\}$. $\lambda_{N}$ is sharp in the sense that if $\lambda> \lambda_{N}$ then there exists a sequence $(u_{n})$ such that $\int_{\Omega}e^{\lambda u^{\frac{N}{N-1}}}dx $ diverges.

The above inequality is related with $N$-Finsler-Laplacian operator $Q_{N}$ which is defined by
       $$Q_{N}u:=\sum_{i=1}^{N}\frac{\partial}{\partial x_{i}}(F^{N-1}(\nabla u)F_{\xi_{i}}(\nabla u)),$$
when $N=2$ and $F(\xi) = |\xi|$, $Q_{2}$ is just the ordinary Laplacian. This operator is closely related to a smooth, convex hypersurface in $\mathbb{R}^{N}$, which is called the Wulff shape (or equilibrium crystal shape) of $F$. This operator $Q_{N}$ was studied in some literatures, see \cite{AFTL,BF,FK} and the references therein. In \cite{ZZ}, they obtained the existence of extremal functions for the sharp geometric inequality \eqref{1.4}. 

Our aim is to establish and find extremal functions for  Trudinger-Moser inequality involving $N$--Finsler--Laplacian and $L^p$ norm.  For  $p >1$, we denote
  $$\gamma_1=\inf_{u\in W^{1, N}_{0}(\Omega), u\not\equiv0}\frac{\int_{\Omega}F^{N}(\nabla u)dx}{\| u\|_p^N},$$
and $$\|u\|_{N,F,\gamma, p}=\bigg(\int_{\Omega}F^{N}(\nabla u)dx-\gamma\| u\|_p^N\bigg)^{\frac{1}{N}}.$$
\begin{theorem}\label{thm1.1}
Let $\Omega\subset \mathbb{R}^{N}(N\geq 2)$ be a smooth bounded domain.  Then for any $0\leq \gamma <\gamma_1,$
\be  \Lambda_{ \gamma}=\sup_{u\in W^{1, N}_{0}(\Omega), \|u\|_{N,F,\gamma, p}\leq 1}\int_{\Omega}e^{\lambda_N |u|^{\frac{N}{N-1}}}dx<+\infty, \lb{1.5}\ee
where $\lambda_{N}=N^{\frac{N}{N-1}} \kappa_{N}^{\frac{1}{N-1}}$ and $\kappa_{N}$ is the volume of a unit Wulff ball.
\end{theorem}
\noindent{\bf Remark 1.2.} From Theorem \ref{thm1.1}, for $0\leq \gamma <\gamma_1,$ we can derive the following weak version:
 \be  \sup_{u\in W^{1, N}_{0}(\Omega), \int_{\Omega}F^N(\nabla u)dx\leq 1}\int_{\Omega}e^{\lambda_N (1+\gamma\| u\|_p^N)^{\frac{1}{N-1}}|u|^{\frac{N}{N-1}}}dx<+\infty, \lb{1.6}\ee
for the special case  $p=N$ in \eqref{1.6}, we refer reader to \cite{Zhou}.
Next, let's show that we have got better result. Indeed, if $0\leq \gamma <\gamma_1,$  $u\in W_0^{1, N}(\Omega)$ and $\int_{\Omega}F^N(\nabla u)dx\leq1$. Set $v=(1+\gamma\| u\|_p^N)^\frac{1}{N}u\in W_0^{1, N}(\Omega)$, note that  $F$ is a  positively homogeneous function of degree 1.   Thus
\begin{equation*}
\left.
\begin{aligned}[b]
 \|v\|^N_{N,F,\gamma, p}&=(1+ \gamma \| u\|_p^N)\int_{\Omega}F^N(\nabla u)dx-\gamma\| u\|_p^N-\gamma^2\| u\|_p^{2N}\\
  &\leq\int_{\Omega}F^N(\nabla u)dx\leq 1.
 \end{aligned}
\right.
\end{equation*}
 Applying \eqref{1.5} to the function $v$, we obtain \eqref{1.6}.  This implies that \eqref{1.5} is a stronger inequality. 
\begin{theorem}\label{thm1.2}
 Let $\Omega\subset \mathbb{R}^{N}(N\geq 2)$ be a smooth bounded domain.  Then
the supremum
   \be   \Lambda_{\gamma}= \sup_{u\in W^{1, N}_{0}(\Omega), \|u\|_{N,F, \gamma, p}\leq 1}\int_{\Omega}e^{\lambda_{N} |u|^{\frac{N}{N-1}}}dx   \lb{1.7}\ee
can be attained by $u_{0}\in W^{1, N}_{0}(\Omega) \bigcap C^{1}(\overline{\Omega})$ with $\|u_0\|_{N,F,\gamma, p}=1.$
\end{theorem}

This paper is organized as follows: In Section \ref{section 2}, we give some
preliminaries, meanwhile, under anisotropic Dirichlet norm and $L^p$ norm,  we establish the Lions type concentration-compactness principle of
 Trudinger-Moser Inequalities. In Section \ref{section 3}, we  give the existence of subcritical maximizers. In Section \ref{section 4}, we analyze the convergence of maximizing sequence  and its blow-up behavior, an upper bound   is established by capacity estimates. In Section \ref{section 5}, we provide the proof of Theorems \ref{thm1.1} and \ref{thm1.2} by contradiction arguments and the construction of test function.
\section{Preliminaries }\label{section 2}
In this section, we will give some preliminaries for our use later. Let $F: \mathbb{R}^{N} \rightarrow [0, +\infty)$ be a  convex function of class $C^{2}(\mathbb{R}^{N} \backslash \{0\})$ which is even and positively homogeneous of degree 1, so that
\be    F(t \xi)=|t|F(\xi)~~~~~\operatorname{for~~any}~~t\in \mathbb{R},~~ \xi \in \mathbb{R}^{N}. \lb{2.1}\ee
We also  assume that $  F(\xi)>0$ for any $\xi \neq 0$  and  $Hess(F^2)$ is positive definite in $\mathbb{R}^{N} \backslash \{0\}$. Then by
Xie and Gong \cite{XG}, $Hess(F^ N)$ is also positive definite in $\mathbb{R}^{N} \backslash \{0\}$. A typical example is $F(\xi)=(\sum_{i}|\xi_i|^q)^{\frac{1}{q}}$ for $q\in [1, \infty)$. Let $F^{o}$ be the support function of $K:=\{x\in \mathbb{R}^{N}: F(x) \leq 1 \}$, which is defined by
      $$F^{o}(x):= \sup_{\xi \in K}\langle x, \xi \rangle, $$
so $F^{o}: \mathbb{R}^{N} \rightarrow [0, +\infty)$ is also a convex, homogeneous function of class  $C^{2}(\mathbb{R}^{N} \backslash \{0\})$. From \cite{AFTL}, $F^{o}$ is dual to $F$ in the sense that
       $$F^{o}(x)= \sup_{\xi \neq 0}\frac{\langle x, \xi \rangle}{F(\xi)},~~~~F(x)= \sup_{\xi \neq 0}\frac{\langle x, \xi \rangle}{F^o(\xi)}.$$

Consider the map $\phi: S^{N-1} \rightarrow \mathbb{R}^{N}$, $\phi(\xi)=F_\xi(\xi)$. Its image $\phi(S^{N-1})$
is smooth, convex hypersurface in $\mathbb{R}^{N}$, which is called the Wulff shape (or equilibrium crystal shape) of $F$. Then $\phi(S^{N-1})= \{x\in \mathbb{R}^{N}|F^o(x) = 1\}$(see \cite{WX1}, Proposition 2.1). Denote $\mathcal{W}_r(x_0)=\{x\in \mathbb{R}^{N}: F^o(x-x_0) \leq r\}$, we call $\mathcal{W}_r(0)$ is a Wulff ball of radius $r$
with center at 0. We will use the convex symmetrization, which is defined in \cite{AFTL}. The convex symmetrization generalizes the Schwarz symmetrization(see \cite{T}). Let us consider a measured function $u$ on $\Omega \subset \mathbb{R}^{N}$, one dimensional decreasing rearrangement of $u$ is
   \be u^{\sharp}(t)=\sup\{s\geq 0: |\{x \in \Omega:   u(x)\geq s \}|>t \}~~~\operatorname{for}~~t\in  \mathbb{R}. \lb{2.2}\ee
The convex symmetrization of $u$ with respect to $F$ is defined as
   \be u^{\star}(x)=u^{\sharp}(\kappa_N F^{o}(x)^{N}) ~~~\operatorname{for}~~x\in \Omega^{\star}.  \lb{2.3}\ee
Here $\kappa_N F^{o}(x)^{N}$ is just the Lebesgue measure of a homothetic Wulff ball with radius $F^{o}(x)$ and $\Omega^{\star}$ is the homothetic Wulff ball centered at the origin having the same measure as $\Omega$. In \cite{AFTL}, the authors proved a P$\acute{o}$lya-Szeg$\ddot{o}$ principle and a comparison result for solutions of the
Dirichlet problem for elliptic equations for the convex symmetrization, which generalizes the classical
results for Schwarz symmetrization due to Talenti \cite{T}.

Now,  we give the definition of  anisotropic perimeter of a set with respect to $F$ , a co-area
formula and an isoperimetric inequality.
Precisely,  for a  a subset $E\subset \Omega$ and a function of bounded variation $u \in BV(\Omega)$, anisotropic bounded variation of $u$ with respect to $F$ is 
 $$\int_\Omega|\nabla u|_F=\sup \left\lbrace  \int_\Omega u     
 \cdot\mathrm{div}\sigma dx: \sigma \in C_0^1(\Omega; \mathbb{R}^N),F^o(\sigma) \leq 1\right\rbrace . $$
Define the  anisotropic perimeter of $E$ with respect to $F$ as
    $$ P_F(E):=\int_\Omega |\nabla \chi_E|_F, \label{pr} $$
where $ \chi_E$ is the characteristic function of the set $E$. From the reference \cite{FM}, we have the co-area formula
       \be \int_\Omega|\nabla u|_F=\int_0^\infty P_F(|u|>t)dt \label{co-area }\ee
 and the isoperimetric inequality
       \be P_F(E)\geq N\kappa_N^{1/N}|E|^{1-\frac{1}{N}}  \label{isoperimetric}.\ee 

 We will establish the Lions type concentration-compactness principle \cite{Lions} for
Trudinger-Moser Inequalities under  anisotropic Dirichlet norm and $L^p$ norm, which is the extention of Theorem 1.1 in \cite{CCH} and Lemma 2.3 in \cite{ZZ}.
\begin{lemma} \label{lem2.1}
Suppose $0\leq \gamma <\gamma_1$. Let $\{u_n\}\subset W_0^{1, N}(\Omega)$ be a sequence such that $\|u\|_{N,F, \gamma, p}=1$, $u_n \rightharpoonup u \not \equiv 0$ weakly in $W_0^{1, N}(\Omega)$. Then for any
    $$0 < q < q_N(u):=(1-\|u\|_{N,F, \gamma, p}^N)^{-\frac{1}{N-1}}$$
we have
  \be \int_{\Omega}e^{\lambda_{N}q |u_n|^{\frac{N}{N-1}}}dx<+\infty  \lb{2.4}\ee
where $\lambda_{N}=N^{\frac{N}{N-1}} \kappa_{N}^{\frac{1}{N-1}}$ and $\kappa_{N}$ is the volume of a unit Wulff ball. Moreover, this conclusion fails if $p \geq p_N(u)$.
\end{lemma}
\begin{proof}[\textbf{Proof.}]  
Since $\|u\|_{N,F, \gamma, p}=1$, we have
   $$\lim_{n\rightarrow \infty}\int_{\Omega}F^{N}(\nabla u_n)=\lim_{n\rightarrow \infty}(1+\gamma\| u_n\|_p^N)=1+\gamma\| u\|_p^N.$$
Let 
$$v_n=\frac{u_n}{(\int_{\Omega}F^{N}(\nabla u_n)dx)^{\frac{1}{N}}},$$
 we have $\int_{\Omega}F^{N}(\nabla v_n)dx=1$ and $v_n\rightharpoonup v:=u/(1+\gamma\| u\|_p^N)^{\frac{1}{N}}$ weakly in $W_0^{1, N}(\Omega)$. By Lemma 2.3 of \cite{ZZ}, it holds
  \be \limsup_{n\rightarrow \infty}\int_{\Omega}e^{\lambda_{N}r |v_n|^{\frac{N}{N-1}}}dx<+\infty  \lb{2.5}\ee
for any $0<r<(1-\int_{\Omega}F^{N}(\nabla v_n)dx)^{-\frac{1}{N-1}}$. Since $q < q_N(u)=(1-\|u\|_{N,F, \gamma, p}^N)^{-\frac{1}{N-1}}$, we have
\begin{equation}
\left.
\begin{aligned}[b]
 \lim_{n\rightarrow \infty}q \int_{\Omega}F^{N}(\nabla u_n)dx)^{\frac{1}{N-1}}&=q(1+\gamma\| u\|_p^N)^{\frac{1}{N-1}}\\
     &<\bigg(\frac{1+\gamma\| u\|_p^N}{1-(\int_{\Omega}F^{N}(\nabla u)dx+\gamma\| u\|_p^N)}\bigg)^{\frac{1}{N-1}}\\
     &=(1-\int_{\Omega}F^{N}(\nabla v_n)dx)^{-\frac{1}{N-1}}
\end{aligned}
\right.  \lb{2.6}
\end{equation}
Take $r<(1-\int_{\Omega}F^{N}(\nabla v_n)dx)^{-\frac{1}{N-1}}$ such that $q \int_{\Omega}F^{N}(\nabla u_n)dx)^{\frac{1}{N-1}}<r$ for $n$ large enough. Thus
 \be \int_{\Omega}e^{\lambda_{N}q |u_n|^{\frac{N}{N-1}}}dx=\int_{\Omega}e^{\lambda_{N}q(\int_{\Omega}F^{N}(\nabla u_n)dx)^{\frac{1}{N-1}} |v_n|^{\frac{N}{N-1}}}dx<\int_{\Omega}e^{\lambda_{N}r|v_n|^{\frac{N}{N-1}}}dx  \lb{2.7}\ee
for $n$ large enough. The result is followed from \eqref{2.5}. \end{proof}

Now, as the similar process of Theorem 2.2 in \cite{Struwe}, we give the following estimate  involving N-Finsler-Laplacian and $L^p$ norm.
\begin{lemma} \label{lem2.2}
Assume that $p>1$ and $0\leq \gamma < \gamma_1$. Let $f \in L^1(\Omega)$ and $u \in C^1(\bar{\Omega}) \bigcap W_0^{1, N}(\Omega)$ satisfies
\begin{equation}\label{2.8}
-Q_Nu=f+\gamma \|u\|_p^{N-p} |u|^{p-2}u~~~   in ~~~ \Omega,
\end{equation}
where $Q_Nu=\mathrm{div}(F^{N-1}(\nabla u)F_{\xi_{i}}(\nabla u))$.
Then  for any $1<q<N$, $u\in W_0^{1, q}(\Omega)$  and
      $$\|u\|_{W_0^{1, q}(\Omega)}\leq C(q, N, \gamma, \gamma_1)\|f\|_{L^1(\Omega)}.$$
\end{lemma}
\begin{proof}[\textbf{Proof}]
 Fix $t > 0$. Testing \eqref{2.8} by  $u^t:=\min\{u, t\} \in W_0^{1, N}(\Omega)$
and integrating by parts, we have
\begin{equation}
\left.
\begin{aligned}[b]
\int_{\Omega}F^N(\nabla u^t)dx\leq& \int_{\Omega}f u^tdx+\gamma \int_{\Omega} \|u^t\|_p^{N-p}|u^t|^pdx\\
 \leq & t\|f\|_{L^1(\Omega)}+\frac{\gamma}{\gamma_1}\int_{\Omega}F^N(\nabla u^t)dx.
\end{aligned}
\right.  \lb{2.9}
\end{equation}
Thus
\begin{equation} \label{2.10}
     \int_{\Omega}F^N(\nabla u^t)dx\leq \frac{\gamma_1}{\gamma_1-\gamma}t\|f\|_{L^1(\Omega)}.
\end{equation}
Denote $\mathcal{W}_r(0)=\{x\in \mathbb{R}^{N}: F^o(x) \leq r\}$ be a Wulff ball of the same measure as $\Omega$.  Let $v^{\star}$ be the convex symmetrization of $u^t$ with respect to $F$ and $|\mathcal{W}_\rho(0)|=|x\in\mathcal{W}_r(0): v^{\star} \geq t\}|$. In \cite{AFTL}, the authors proved the P$\acute{o}$lya-Szeg$\ddot{o}$ principle
      $$\int\limits_{\mathcal{W}_r(0)}F^N(\nabla v^{\star})dx\leq \int_{\Omega}F^N(\nabla u^t)dx.  $$
Thus
\begin{equation}
\left.
\begin{aligned}[b]
   \inf\limits_{\phi \in W_0^{1, N}(\mathcal{W}_r(0)), \phi|\mathcal{W}_\rho=t}\int\limits_{\mathcal{W}_r(0)}F^N(\nabla \phi)dx  \leq & \int\limits_{\mathcal{W}_r(0)}F^N(\nabla v^{\star})dx\\
 \leq & \frac{\gamma_1}{\gamma_1-\gamma}t\|f\|_{L^1(\Omega)}.
\end{aligned}
\right.  \lb{2.11}
\end{equation}
On the other hand,  the above infimum can be  achieved by
 \begin{equation}
  \phi_0(x)= \left\{  \begin{array}{l}
          t\log\frac{r}{F^o(x)}/\log\frac{r}{\rho}~~~   in~~~~   \mathcal{W}_r(0)\backslash \mathcal{W}_\rho(0), \\
         t   ~~~~~~~~~~~~~~~~~~~~~~~   in~~~~    \mathcal{W}_\rho(0).
         \end{array}
   \right.   \lb{2.12}
 \end{equation}
Since $F(\nabla F^o(x))=1$, through the direct computation, we have
\begin{equation}
\left.
\begin{aligned}[b]
\int\limits_{\mathcal{W}_r(0)}F^N(\nabla \phi_0)dx=&\int_\rho^r F^N( \frac{t}{\log\frac{r}{\rho}} \frac{\nabla F^o(x)}{-s})\int_{\partial \omega_s}\frac{1}{|\nabla F^o(x)|}d\sigma ds\\
=&\int_\rho^r  \frac{t^N}{(\log\frac{r}{\rho})^N} \frac{1}{s^N}\int_{\partial \omega_s}\frac{1}{|\nabla F^o(x)|}d\sigma ds\\
=&\int_\rho^r  \frac{t^N}{(\log\frac{r}{\rho})^N} \frac{1}{s^N}N\kappa_{N} s^{N-1} ds \\
=& \frac{N\kappa_{N}t^N}{(\log\frac{r}{\rho})^{N-1}}.
\end{aligned}
\right.  \lb{2.13}
\end{equation}
Hence $\frac{N\kappa_{N}t^N}{(\log\frac{r}{\rho})^{N-1}} \leq \frac{\gamma_1}{\gamma_1-\gamma}t\|f\|_{L^1(\Omega)}$, it holds
\begin{equation}
\left.
\begin{aligned}[b]
|\{x\in \Omega : u(x)\geq t\}|=|\mathcal{W}_\rho(0)|=&\kappa_{N} \rho^N \\
\leq& \kappa_{N} r^N\exp(-N(N\kappa_{N})^{\frac{1}{N-1}}t(\frac{\gamma_1-\gamma}{\gamma_1}\|f\|_{L^1}^{-1})^{\frac{1}{N-1}})\\
 \leq& |\Omega|\exp(-N(N\kappa_{N})^{\frac{1}{N-1}}t(\frac{\gamma_1-\gamma}{\gamma_1}\|f\|_{L^1}^{-1})^{\frac{1}{N-1}})\\
 =&: |\Omega|e^{-NC_1t}.
\end{aligned}
\right.  \lb{2.14}
\end{equation}
For every $0<b<NC_1$,
\begin{equation}
\left.
\begin{aligned}[b]
\int_\Omega e^{bu}dx \leq& \int_{\{x: u(x)\leq 1\}} e^{bu}dx+\int_{\{x: u(x)\geq 1\}} e^{bu}dx\\
 \leq& e^b|\Omega|+\sum_{k=1}^{\infty}e^{b(k+1)}|\{x\in \Omega: k\leq u\leq k+1\}|\\
 \leq& e^b|\Omega|+e^b|\Omega|\sum_{k=1}^{\infty}e^{(b-NC_1)k}\leq C(b)|\Omega|.
\end{aligned}
\right.  \lb{2.15}
\end{equation}
From \eqref{2.10}, we have
\begin{equation} \label{2.16}
     \int_{\{x:  u\leq t\}}F^N(\nabla u)dx\leq \frac{\gamma_1}{\gamma_1-\gamma}t\|f\|_{L^1(\Omega)}.
\end{equation}
Hence,
\begin{equation}
\left.
\begin{aligned}[b]
\int_\Omega \frac{F^N(\nabla u)}{1+u^2}dx =&  \int_{\{x: u(x)\leq 1\}}\frac{F^N(\nabla u)}{1+u^2}dx+\int_{\{x: u(x)\geq 1\}} \frac{F^N(\nabla u)}{1+u^2}dx\\
     \leq& \int_{\{x: u(x)\leq 1\}}F^N(\nabla u)dx+\sum_{m\geq 0}\int_{\{x: 2^m\leq u(x)\leq 2^{m+1}\}} \frac{F^N(\nabla u)}{u^2}dx\\
     \leq& \int_{\{x: u(x)\leq 1\}}F^N(\nabla u)dx+\sum_{m\geq 0}\frac{1}{2^m}\int_{\{x: 2^m\leq u(x)\leq 2^{m+1}\}} \frac{F^N(\nabla u)}{u}dx\\
 \leq&\frac{\gamma_1}{\gamma_1-\gamma}\|f\|_{L^1(\Omega)}+\sum_{m\geq 0}\frac{1}{2^m}\frac{2\gamma_1}{\gamma_1-\gamma}\|f\|_{L^1(\Omega)}\\
 =&\frac{5\gamma_1}{\gamma_1-\gamma}\|f\|_{L^1(\Omega)},
\end{aligned}
\right.  \lb{2.17}
\end{equation}
where we have used the estimate \eqref{2.16} for $t = 1$ and $t=2^{m+1}$ in last inequality.
For $1<q<N$, $\frac{q}{N}+\frac{N-q}{N}=1$, by Young's inequality, we have
\begin{equation}
\left.
\begin{aligned}[b]
\int_\Omega F^q(\nabla u)dx =& \int_\Omega \frac{F^q(\nabla u)}{(1+u^{2})^{q/N}}(1+u^{2})^{q/N}dx\\
                  \leq & \int_\Omega \frac{F^N(\nabla u)}{1+u^2}dx+\int_\Omega (1+u^{2})^{q/(N-q)}dx\\
\end{aligned}
\right.  \lb{2.18}
\end{equation}
The desired bound now follows from \eqref{2.15} and \eqref{2.17}. 
\end{proof}
\section{Maximizers of the subcritical case}\label{section 3}
In this section, we will show the existence of the maximizers for Trudinger-Moser in the subcritical case. 
\begin{lemma}\label{lem3.1}
 Let $\Omega\subset \mathbb{R}^{N}(N\geq 2)$ be a smooth bounded domain.  Then for any $\epsilon \in (0, \lambda_{N})$, the supremum
     \be   \Lambda_{\gamma, \epsilon} =\sup_{u\in W^{1, N}_{0}(\Omega),\|u\|_{N,F, \gamma, p}\leq 1}\int_{\Omega}e^{(\lambda_{N}-\epsilon)
     |u|^{\frac{N}{N-1}}}dx   \lb{3.0}\ee
can be attained by $u_{\epsilon}\in W^{1, N}_{0}(\Omega) \bigcap C^{1}(\overline{\Omega})$ with $\|u _{\epsilon, n}\|_{N,F, \gamma, p}= 1$ . In the distributional sense, $u_\epsilon$ satisfies the following equation
 \begin{equation}
   \left\{  \begin{array}{l}
          -Q_{N}u_\epsilon-\gamma\| u_\epsilon\|_p^{N-p}u_\epsilon^{p-1}=\frac{1}{\lambda_\epsilon}u_\epsilon^{\frac{1}{N-1}}e^{(\lambda_{N}-\epsilon)u_\epsilon^{\frac{N}{N-1}}}      ~~~~~~in~~~\Omega,\\
          u_{\epsilon}=0  ~~~~~on~~~\partial\Omega, \\
         \lambda_\epsilon=\int_\Omega u_\epsilon^{\frac{N}{N-1}}e^{(\lambda_{N}-\epsilon)u_\epsilon^{\frac{N}{N-1}}}dx.
         \end{array}
   \right.   \lb{3.1}
 \end{equation}
 Moreover,
     $$\liminf_{\epsilon\rightarrow 0}\lambda _{\epsilon}>0$$
 and
 \begin{equation}\label{3.2}    
      \liminf_{\epsilon\rightarrow 0}\Lambda_{\gamma, \epsilon}=\Lambda_{\gamma}.
    \end{equation}
\end{lemma} 
\begin{proof}[\textbf{Proof}]
Let  $u_{\epsilon, n}$ be a maximizing sequence for $\Lambda_{\gamma, \epsilon}$, \emph{i.e.},   $u_{\epsilon, n}\in W_0^{1, N}(\Omega)$, $\|u _{\epsilon, n}\|_{N,F, \gamma, p}\leq 1$ and
    \be   \int_{\Omega}e^{(\lambda_{N}-\epsilon)
     |u_{\epsilon, n}|^{\frac{N}{N-1}}}dx \rightarrow \Lambda_{\gamma, \epsilon}  \lb{3.3}\ee
as $n\rightarrow \infty$. Since $\gamma<\gamma_1$, thus
   $$(1-\frac{\gamma}{\gamma_1})\int_{\Omega}F^{N}(\nabla u_{\epsilon, n})dx\leq\int_{\Omega}F^{N}(\nabla u_{\epsilon, n})dx-\gamma \| u_{\epsilon, n}\|_p^N\leq 1,$$
which lead to $u_{\epsilon, n}$ is bounded in $W^{1, N}_{0}(\Omega)$, so there exists some $u_\epsilon \in W^{1, N}_{0}(\Omega)$
such that up to a subsequence, $u_{\epsilon, n}\rightharpoonup u_{\epsilon}$ weakly in $W_0^{1, N}(\Omega)$, $u_{\epsilon, n}\rightarrow u_{\epsilon}$ strongly in
$L^q(\Omega)$ for any $q \geq 1$, and $u_{\epsilon, n}\rightarrow u_{\epsilon}$ a. e. in $\Omega$. We claim that $u_\epsilon \not\equiv 0$.   If
otherwise, then $\limsup\limits_{n\rightarrow \infty}\int_{\Omega}F^{N}(\nabla u_{\epsilon, n})dx\leq 1$. The anisotropic  Trudinger-Moser inequality implies  $e^{(\lambda_{N}-\epsilon)|u_{\epsilon, n}|^{\frac{N}{N-1}}}$ is uniformly bounded in $L^s(\Omega)$ for any $1<s<\frac{\lambda_{N}}{\lambda_{N}-\epsilon}$. Thus
  \be   \Lambda_{\gamma, \epsilon}=\lim_{n\rightarrow \infty} \int_{\Omega}e^{(\lambda_{N}-\epsilon)
     |u_{\epsilon, n}|^{\frac{N}{N-1}}}dx =|\Omega|,  \lb{3.4}\ee
which is impossible.  So  $u_\epsilon \not\equiv 0$ and  Lemma \ref{lem2.1} implies $\int_{\Omega}e^{(\lambda_{N}-\epsilon)
     |u_{\epsilon, n}|^{\frac{N}{N-1}}}dx$  in $L^s(\Omega)$ for some $s>1$ . Consequently, we have
 \be   \Lambda_{\gamma, \epsilon}=\lim_{n\rightarrow \infty} \int_{\Omega}e^{(\lambda_{N}-\epsilon)
     |u_{\epsilon, n}|^{\frac{N}{N-1}}}dx =\int_{\Omega}e^{(\lambda_{N}-\epsilon)
     |u_{\epsilon}|^{\frac{N}{N-1}}}dx,  \lb{3.5}\ee
Then  $u_\epsilon$ attains the supremum. We claim $\|u _{\epsilon}\|_{N,F, \gamma, p}=1$. In fact, if $\|u _{\epsilon}\|_{N,F,\gamma,p}<1$, then
  \be \Lambda_{\gamma, \epsilon}=\int_{\Omega}e^{(\lambda_{N}-\epsilon)
     |u_{\epsilon}|^{\frac{N}{N-1}}}dx < \int_{\Omega}e^{(\lambda_{N}-\epsilon)
     |\frac{u_{\epsilon}}{\|u _{\epsilon}\|_{N,F, \gamma, p}}|^{\frac{N}{N-1}}}dx \leq \Lambda_{N, \epsilon},  \lb{3.6}\ee
which is a contradiction. Furthermore, we know that $u_\epsilon$ satisfies the Euler--Lagrange equation in distributional sense.
By regularity theory  obtained in \cite{Lieberman},  we have  $u_\epsilon \in C^{1}(\overline{\Omega})$.

  Since $e^t\leq 1+te^t$ for $t\geq 0$, we have
 \begin{equation*}
\left.
\begin{aligned}[b]
  &\int_{\Omega}e^{(\lambda_{N}-\epsilon)
     |u_{\epsilon}|^{\frac{N}{N-1}}}dx\\
     \leq & |\Omega|+(\lambda_{N}-\epsilon)\int_{\Omega}|u_{\epsilon}|^{\frac{N}{N-1}}e^{(\lambda_{N}-\epsilon)
     |u_{\epsilon}|^{\frac{N}{N-1}}}dx\\
     = & |\Omega|+(\lambda_{N}-\epsilon)\lambda_{\epsilon}.
\end{aligned}
\right.
\end{equation*}
This leads to $\liminf_{\epsilon\rightarrow 0}\lambda _{\epsilon}>0.$

Obviously, $\limsup_{\epsilon\rightarrow 0}\Lambda _{\gamma, \epsilon}\leq \Lambda _{\gamma}$. On the other hand, for any $u\in W_0^{1, N}(\Omega)$ with $\|u _{\epsilon}\|_{N,F,\gamma,p}\leq1$, by  Fatou's lemma, we have
      $$\int_{\Omega}e^{\lambda_{N}
     |u|^{\frac{N}{N-1}}}dx\leq \liminf_{\epsilon\rightarrow 0} \int_{\Omega}e^{(\lambda_{N}-\epsilon)
     |u_{\epsilon}|^{\frac{N}{N-1}}}dx \leq  \liminf_{\epsilon\rightarrow 0}\Lambda _{\gamma, \epsilon}, $$
which implies that  $\liminf_{\epsilon\rightarrow 0}\Lambda _{\gamma, \epsilon} \geq \Lambda _{\gamma}$. Thus  $\liminf_{\epsilon\rightarrow 0}\Lambda_{\gamma, \epsilon}=\Lambda_{\gamma}.$
\end{proof}
\section{Maximizers of the critical case}\label{section 4}
In this section, by using blow-up analysis, we analyze the behavior of the maximizers $u_\epsilon$ in section 3. Since $u_\epsilon$ is bounded in $W^{1, N}_{0}(\Omega)$,
up to a subsequence, we can assume  $u_{\epsilon}\rightharpoonup u_{0}$ weakly in $W_0^{1, N}(\Omega)$, $u_{\epsilon}\rightarrow u_{0}$ strongly in
$L^q(\Omega)$ for any $q \geq 1$, and $u_{\epsilon}\rightarrow u_{0}$ a.e. in $\Omega$ as $\epsilon\rightarrow 0$.
\subsection{Blow-up analysis} \label{subsection 4.1}
Let $c_\epsilon = \max\limits_{\overline{\Omega}} u_\epsilon=u_\epsilon(x_\epsilon)$. If $c_\epsilon$ is bounded, then for any $u \in W^{1, N}_{0}(\Omega)$ with $\|u \|_{N,F,\gamma,p}\leq 1$,
by Lebesgue dominated convergence theorem,  we have
\begin{equation}
\left.
\begin{aligned}[b]
     \int_\Omega e^{\lambda_{N}|u|^{\frac{N}{N-1}}}dx=&\lim_{\epsilon\rightarrow 0} \int_\Omega e^{(\lambda_{N}-\epsilon)|u|^{\frac{N}{N-1}}}dx\\
     \leq&\lim_{\epsilon\rightarrow 0} \int_\Omega e^{(\lambda_{N}-\epsilon)|u_\epsilon|^{\frac{N}{N-1}}}dx\\
     =&\int_\Omega e^{\lambda_{N}|u_0|^{\frac{N}{N-1}}}dx
\end{aligned}
\right.    \lb{4.1}
\end{equation}
Therefore $u_0$ is the desired maximizer. Moreover, $u_\epsilon \rightarrow u_0$ and  $u_0 \in C^1(\overline{\Omega})$ by standard elliptic regularity
theory. In the following, we consider another case, we assume $c_\epsilon \rightarrow +\infty$ and $x_\epsilon\rightarrow x_0$ as $\epsilon \rightarrow 0$. We assume $x_0\in \Omega$, at the end of this section, we shall exclude the case $x_0\in \partial\Omega$.
\begin{lemma} \label{lem4.1}
$u_0 \equiv 0$ and  $ F^N(\nabla u_\epsilon)dx\rightharpoonup \delta_{x_0}$ weakly in the sense of measure as $\epsilon \rightarrow 0$, where $\delta_{x_0}$ is the Dirac measure at $x_0$.
\end{lemma}
\begin{proof}[\textbf{Proof}]
Suppose $u_0 \not \equiv 0$. Notice that $\liminf_{\epsilon\rightarrow 0}\lambda _{\epsilon}>0$, by  Lemma \ref{lem2.1} and H$\ddot{o}$lder inequality, we have
$\frac{1}{\lambda_\epsilon}u_\epsilon|u_\epsilon|^{\frac{2-N}{N-1}}e^{(\lambda_{N}-\epsilon)|u_\epsilon|^{\frac{N}{N-1}}}$ is uniformly bounded in $L^q(\Omega)$ for some $q>1$. Then by Lemma 2.2 in \cite{ZZ}, $u_\epsilon$ is uniformly bounded in $\Omega$, which contradicts $c_\epsilon\rightarrow +\infty$ as $\epsilon \rightarrow 0$. Hence $u_0\equiv 0$.

Notice that $\int_{\Omega}F^N(\nabla u_\epsilon)dx=1+\gamma \| u_\epsilon\|_p^{N}\rightarrow 1$. If $ F^N(\nabla u_\epsilon)dx\rightharpoonup \mu\neq \delta_{x_0}$ in the sense of measure as $\epsilon \rightarrow 0$, then there exists $\theta <1$ and $r>0$ small enough such that
           $$\lim_{\epsilon \rightarrow 0}\int\limits_{\mathcal{W}_r(x_0)}F^N(\nabla u_\epsilon)dx<\theta,$$
 Consider the cut-off function $\phi\in C_0^\infty(\Omega)$, which is supported in $\mathcal{W}_r(x_0)$ for some $r>0$, $0 \leq\phi \leq 1$  and $\phi = 1$ in $\mathcal{W}_{\frac{r}{2}}(x_0)$. Since $u_\epsilon\rightarrow 0$ in $L^q(\Omega)$ for any $q>1$, we have
           $$\limsup_{\epsilon \rightarrow 0}\int_{\mathcal{W}_r(x_0)}F^N(\nabla (\phi u_\epsilon)dx\leq \lim_{\epsilon \rightarrow 0}\int_{\mathcal{W}_r(x_0)}F^N(\nabla  u_\epsilon dx<\theta.$$
 For sufficiently small $\epsilon > 0$, using anisotropic Moser-Trudinger inequality to $\phi u_\epsilon$,we have $e^{\lambda_N u_{\epsilon}^{\frac{N}{N-1}}}dx$
 is uniformly bounded in $L^{s}(\mathcal{W}_{\frac{r}{2}}(x_0))$ for some $s>1$. From  Lemma 2.2 in  \cite{ZZ},  $u_\epsilon$ is  uniformly bounded in $\mathcal{W}_{\frac{r}{2}}(x_0)$,  which contradicts to $c_\epsilon\rightarrow +\infty$. Hence $ F^N(\nabla u_\epsilon)dx\rightharpoonup \delta_{0}$ as $\epsilon\rightarrow 0$.      
 \end{proof}

Let
   \be r_\epsilon^N=\lambda_\epsilon c_\epsilon^{-N/(N-1)}e^{-(\lambda_{N}-\epsilon)c_\epsilon^{N/(N-1)}}.  \label{4.2}\ee
Denote $\Omega_\epsilon=\{x\in \mathbb{R}^{N}: x_\epsilon+r_\epsilon x\in \Omega\}$. Define
      \be \psi_{ \epsilon}(x)= c_\epsilon^{-1}u_\epsilon(x_\epsilon+r_\epsilon x)  \lb{4.3}\ee
and
   \be \varphi_{ \epsilon}(x)= c_\epsilon^{N/(N-1)}(\psi_{ \epsilon}(x)-1)=c_\epsilon^{1/(N-1)}(u_\epsilon(x_\epsilon+r_\epsilon x)-c_\epsilon).  \lb{4.4}\ee
We have the following:
\begin{lemma} \label{lem4.2}
Let $r_\epsilon, \psi_{ \epsilon}(x)$ and $\varphi_{\epsilon}$ be defined as in \eqref{4.2}-\eqref{4.4}. Then\\
  (i) there hold $r_\epsilon \rightarrow 0$ as $\epsilon\rightarrow 0$;\\
  (ii) $\psi_{ \epsilon}(x)\rightarrow 1$ in $C_{loc}^1( \mathbb{R}^{N})$;\\
  (iii)  $\varphi_{ \epsilon}(x)\rightarrow\varphi (x)$ in $C_{loc}^1( \mathbb{R}^{N})$, where
                            \be \varphi (x)=-\frac{N-1}{\lambda_N}\log\bigg(1+ \kappa_{N}^{\frac{1}{N-1}}F^{o}(x)^{\frac{N}{N-1}}\bigg).  \lb{(4.5)}\ee
Moreover,
     \be \int_{\mathbb{R}^{N}}e^{\frac{N}{N-1}\lambda_N\varphi}dx=1.  \lb{(4.6)}\ee
\end{lemma}
\begin{proof}[\textbf{Proof}]
(i) From $\eqref{4.2}$, we have
\begin{equation*}
\left.
\begin{aligned}[b]
    r_\epsilon^Nc_\epsilon^{N/(N-1)}=&\lambda_\epsilon e^{-(\lambda_{N}-\epsilon)c_\epsilon^{N/(N-1)}}\\
    =&\int_\Omega  |u_\epsilon|^{\frac{N}{N-1}}e^{(\lambda_{N}-\epsilon)|u_\epsilon|^{\frac{N}{N-1}}}dx\cdot e^{-(\lambda_{N}-\epsilon)c_\epsilon^{N/(N-1)}}\\
     \leq &  c_\epsilon^{N/(N-1)}e^{(\frac{\lambda_N}{2}-\epsilon)c_\epsilon^{N/(N-1)}}\int_\Omega e^{\frac{\lambda_N}{2}|u_\epsilon|^{\frac{N}{N-1}}}dx\cdot e^{-(\lambda_N-\epsilon)c_\epsilon^{N/(N-1)}}\\
       \leq &Cc_\epsilon^{N/(N-1)}e^{-\frac{\lambda_N}{2}c_\epsilon^{N/(N-1)}}\rightarrow 0,
\end{aligned}
\right.
\end{equation*}
as $\epsilon \rightarrow 0$. Thus $\lim\limits_{\epsilon \rightarrow 0}r_\epsilon = 0$.

(ii) By a direction calculation, $0\leq \psi_{\epsilon} \leq 1$ and $\psi_{\epsilon}$ is a weak solution to
\begin{equation*}
\left.
\begin{aligned}[b]
 &-\mathrm{div}(F^{N-1}(\nabla \psi_{ \epsilon})F_{\xi}(\nabla \psi_{ \epsilon}))\\
 &=c_\epsilon^{-N}e^{(\lambda_{N}-\epsilon)c_\epsilon^{N/(N-1)}(\psi_{\epsilon}^{\frac{N}{N-1}}-1)}\psi_{\epsilon}^{\frac{1}{N-1}}+\gamma c_\epsilon^{p-N} r_\epsilon^N\|u_\epsilon\|_p^{N-p}\psi_{\epsilon}^{p-1}
    ~~~~~~in~~~\Omega_\epsilon.
 \end{aligned}
\right.
\end{equation*}
Since
$c_\epsilon^{-N}e^{(\lambda_{N}-\epsilon)c_\epsilon^{N/(N-1)}(\psi_{\epsilon}^{\frac{N}{N-1}}-1)}\psi_{\epsilon}^{\frac{1}{N-1}}\leq c_\epsilon^{-N}\rightarrow 0$ as $\epsilon \rightarrow 0$,
and
\begin{equation*}
\left.
\begin{aligned}[b]
  &\bigg(\int\limits_{B_{r_\epsilon^{-1}}}( c_\epsilon^{p-N} r_\epsilon^N\|u_\epsilon\|_p^{N-p}\psi_{\epsilon}^{p-1})^{p/(p-1)}dx\bigg)^{\frac{p-1}{p}}\\
= &c_\epsilon^{1-N} r_\epsilon^{\frac{N}{p}}\|u_\epsilon\|_p^{N-1}\rightarrow 0 ~~as~~ \epsilon \rightarrow 0.
 \end{aligned}
\right.
\end{equation*}
Thus we obtain that $\mathrm{div}(F^{N-1}(\nabla \psi_{ \epsilon})F_{\xi}(\nabla \psi_{ \epsilon}))$ is bounded in $L^{p/(p-1)}(B_{r_\epsilon^{-1}})$. Using elliptic regularity theory (see\cite{P.Tolksdorf}), we have $\psi_{\epsilon}\rightarrow \psi$ in $C_{loc}^{1}(\mathbb{R}^{N})$, where $\psi$ is
a weak solution to the equation
 $$-\mathrm{div}(F^{N-1}(\nabla \psi)F_{\xi}(\nabla \psi))=0~~~ in~~~ \mathbb{R}^{N}.$$
The Liouville theorem (see \cite{HKM}) implies $\psi\equiv 1.$

(iii) We have
\begin{equation*}
\left.
\begin{aligned}[b]
 &-\mathrm{div}(F^{N-1}(\nabla \varphi_{ \epsilon})F_{\xi}(\nabla \varphi_{ \epsilon}))\\
 &=e^{(\lambda_{N}-\epsilon)c_\epsilon^{N/(N-1)}(\psi_{\epsilon}^{\frac{N}{N-1}}-1)}\psi_{\epsilon}^{\frac{1}{N-1}}+\gamma c_\epsilon^{p} r_\epsilon^N\|u_\epsilon\|_p^{N-p}\psi_{\epsilon}^{p-1} ~~~~~~in~~~\Omega_\epsilon.
 \end{aligned}
\right.
\end{equation*}
When $p>N$,  for $R > 0$ and sufficiently small $\epsilon$,
\begin{equation*}
\left.
\begin{aligned}[b]
 \|u_\epsilon\|_p^{N-p}\leq&\bigg(\int\limits_{B_{Rr_\epsilon}}u_\epsilon^pdx\bigg)^{N/P-1}\\
                       =&c_\epsilon^{N-p}r_\epsilon^{N^2/p-N}\bigg(\int\limits_{B_{R}}\psi_\epsilon^pdx\bigg)^{N/P-1}
 \end{aligned}
\right.
\end{equation*}
Since $ 0\leq\psi_{\epsilon} \leq 1$, $\psi_{\epsilon} \rightarrow 1$ in $C_{loc}^1( \mathbb{R}^{N})$, we have $\int\limits_{B_{R}}\psi^pdx>0$. Thus
  \begin{equation*}
\left.
\begin{aligned}[b]
 c_\epsilon^pr_\epsilon^N\|u_\epsilon\|_p^{N-p}\psi_{\epsilon}^{p-1}
 \leq&2\bigg(\int\limits_{B_{R}}\psi_\epsilon^pdx\bigg)^{N/P-1}c_\epsilon^Nr_\epsilon^{N^2/p}\rightarrow 0 ~~as~~\epsilon \rightarrow0.
 \end{aligned}
\right.
\end{equation*}
When $1 < p \leq N$, obviously,   $c_\epsilon^pr_\epsilon^N\|u_\epsilon\|_p^{N-p}\psi_{\epsilon}^{p-1}\rightarrow 0$  as $\epsilon \rightarrow0$.
From regularity theory in \cite{P.Tolksdorf}, up to a subsequence, there exists  $\varphi \in C^1(\mathbb{R}^{N})$ such that $ \varphi_{ \epsilon}\rightarrow \varphi$ in $C^{1}_{loc}(\mathbb{R}^{N})$. From the definition of $\psi_\epsilon$ and $\varphi_\epsilon$,  we have
\begin{equation*}
\left.
\begin{aligned}[b]
 u_\epsilon(x_\epsilon+r_\epsilon x)^{\frac{N}{N-1}}-c_\epsilon^{\frac{N}{N-1}}&=c_\epsilon^{\frac{N}{N-1}}\bigg(\psi_\epsilon(x)^{\frac{N}{N-1}}- 1\bigg)\\
 &= c_\epsilon^{\frac{N}{N-1}}\bigg((1+\psi_\epsilon(x)-1)^{\frac{N}{N-1}}- 1\bigg)\\
  &=\frac{N}{N-1}\varphi_{\epsilon}(x)+c_\epsilon^{\frac{N}{N-1}}O((\psi_\epsilon(x)-1)^2)\\
 &=\frac{N}{N-1}\varphi_{\epsilon}(x)+O(\psi_\epsilon(x)-1).
 \end{aligned}
\right.
\end{equation*}
By $\epsilon\rightarrow 0$,  $\varphi$ satisfies
  \begin{equation*}
   \left\{  \begin{array}{l}
          -\mathrm{div}(F^{N-1}(\nabla \varphi)F_{\xi}(\nabla \varphi))=e^{\frac{N}{N-1}\lambda_{N}\varphi} ~~~~~~in~~~\mathbb{R}^{N},\\
         \varphi(x)\leq \varphi(0)=0.
         \end{array}
   \right.   
 \end{equation*}
Combing \eqref{4.2} and \eqref{4.3}, for any $R > 0$,      we have
\begin{equation*}
\left.
\begin{aligned}[b]
  \int\limits_{\mathcal{W}_R(0)}e^{\frac{N}{N-1}\lambda_N\varphi}dx
  =&\lim_{\epsilon \rightarrow 0}\int\limits_{\mathcal{W}_R(0)}e^{(\lambda_N-\epsilon)( u_\epsilon(x_\epsilon+r_\epsilon x)^{\frac{N}{N-1}}-c_\epsilon^{\frac{N}{N-1}})}dx\\
    =&\lim_{\epsilon \rightarrow 0}\lambda_\epsilon^{-1} \int\limits_{\mathcal{W}_Rr_\epsilon(x_\epsilon)}c_\epsilon^{N/(N-1)}e^{(\lambda_N-\epsilon) u_\epsilon(y)^{\frac{N}{N-1}}}dy\\
     =&\lim_{\epsilon \rightarrow 0}\lambda_\epsilon^{-1} \int\limits_{\mathcal{W}_Rr_\epsilon(x_\epsilon)}u_\epsilon(y)^{N/(N-1)}e^{(\lambda_N-\epsilon) u_\epsilon(y)^{\frac{N}{N-1}}}dy\\
     \leq& 1.
 \end{aligned}
\right.
\end{equation*}
This leads to
 $$\int\limits_{\mathbb{R}^{N}}e^{\frac{N}{N-1}\lambda_N\varphi}dx\leq 1.$$
From co-area formula \eqref{co-area } and isoperimetric inequality \eqref{isoperimetric}, through a simple computation, it follows from H\"oder inequlity, 
 we have $\int_{\mathbb{R}^{N}}e^{\frac{N}{N-1}\lambda_N\varphi}dx\geq 1.$
Thus
    \be \int_{\mathbb{R}^{N}}e^{\frac{N}{N-1}\lambda_N\varphi}dx=1,  
    \lb{4.7}\ee
 which implies $\varphi$ is  symmetric with respect to $F^o$, i.e., $\varphi(x)=\varphi(F^o(x))$(see \cite{WX1}, Prop. 6.1). Thus we get
  \begin{equation}\label{4.8}
  \varphi(x)=-\frac{N-1}{\lambda_N}\log\bigg(1+ \kappa_{N}^{\frac{1}{N-1}}F^{o}(x)^{\frac{N}{N-1}}\bigg).
  \end{equation}
\end{proof}

Denote $\mathcal{W}_{R}(x_\epsilon)$ be a Wulff ball of radius $R$
with center at $x_\epsilon$. For any $0 < a < 1$, use the notation
   $$u_{\epsilon, a}=\min\{u_{\epsilon}, a c_\epsilon\}.$$
Then we have the following:
\begin{lemma} \label{lem4.3}
For any $0 < a < 1$, there holds
    $$\lim_{\epsilon \rightarrow 0}\int_\Omega F^N(\nabla u_{\epsilon, a})dx=a.$$
\end{lemma}
\begin{proof}[\textbf{Proof}]
 Testing the equation \eqref{3.2}, we have
\begin{equation*}
\left.
\begin{aligned}[b]
  \int_{\Omega}F^N(\nabla u_{\epsilon, a})dx =&\int_{\Omega}F^{N-1}(\nabla u_{\epsilon})F_{\xi}(\nabla u_{\epsilon})\cdot u_{\epsilon, a}dx\\
   =&-\int_{\Omega}\mathrm{div}(F^{N-1}(\nabla u_{\epsilon})F_{\xi}(\nabla u_{\epsilon}))u_{\epsilon, a}dx\\
   =&\int_{\Omega}\frac{1}{\lambda_\epsilon}u_\epsilon^{\frac{1}{N-1}}e^{(\lambda_{N}-\epsilon)u_\epsilon^{\frac{N}{N-1}}}u_{\epsilon, a}dx+\gamma\int_{\Omega}\|u_\epsilon\|_p^{N-p} u_\epsilon^{p-1}u_{\epsilon, a}dx.
 \end{aligned}
\right.
\end{equation*}
 For any $R > 0$, we have $\mathcal{W}_{R}(x_\epsilon)\subset \{u_\epsilon > ac_\epsilon \}$ for $\epsilon > 0$ small enough. Thus
 \begin{equation*}
\left.
\begin{aligned}[b]
  \int_{\Omega}F^N(\nabla u_{\epsilon, a})dx
   =&\int_{\{u_\epsilon \leq ac_\epsilon\}}\frac{1}{\lambda_\epsilon}u_\epsilon^{\frac{1}{N-1}}e^{(\lambda_{N}-\epsilon)u_\epsilon^{\frac{N}{N-1}}}u_{\epsilon, a}dx\\
   &+\int_{\{u_\epsilon >ac_\epsilon\}}\frac{1}{\lambda_\epsilon}u_\epsilon^{\frac{1}{N-1}}e^{(\lambda_{N}-\epsilon)u_\epsilon^{\frac{N}{N-1}}}u_{\epsilon, a}dx
   +\gamma\int_{\Omega}\|u_\epsilon\|_p^{N-p} u_\epsilon^{p-1}u_{\epsilon, a}dx\\
  >& ac_\epsilon \int\limits_{\mathcal{W}_{R}(x_\epsilon)}\frac{1}{\lambda_\epsilon}u_\epsilon^{\frac{1}{N-1}}e^{(\lambda_{N}-\epsilon)u_\epsilon^{\frac{N}{N-1}}}u_{\epsilon, a}dx+o_\epsilon(1).
 \end{aligned}
\right.
\end{equation*}
Set $ x = x_\epsilon + r_\epsilon y$, we have
\begin{equation*}
\left.
\begin{aligned}[b]
             &ac_\epsilon \int\limits_{\mathcal{W}_{R}(x_\epsilon)}\frac{1}{\lambda_\epsilon}u_\epsilon^{\frac{1}{N-1}}e^{(\lambda_{N}-\epsilon)u_\epsilon^{\frac{N}{N-1}}}u_{\epsilon, a}dx\\
             =&a \int\limits_{\mathcal{W}_{R}(0)}\psi_\epsilon(y)^{\frac{1}{N-1}}e^{(\lambda_{N}-\epsilon)c_\epsilon^{\frac{N}{N-1}}(\psi_\epsilon(y)^{\frac{N}{N-1}}- 1)}dy\rightarrow \int\limits_{\mathcal{W}_{R}(0)}e^{\frac{N}{N-1}\lambda_N\varphi}dy.
\end{aligned}
\right.
\end{equation*}
Letting $R \rightarrow \infty$, we derive that
 $$\liminf\limits_{\epsilon \rightarrow 0}\int_{\Omega}F^N(\nabla u_{\epsilon, a})dx\geq a.$$
Similarly, we choose $(u_{\epsilon}- u_{\epsilon, a})$ as a test function of  \eqref{3.2}, we obtain $\liminf\limits_{\epsilon \rightarrow 0}\int_{\Omega}F^N(\nabla( u_{\epsilon}- u_{\epsilon, a}))\geq 1-a$. Notice that
 \begin{equation*}
\left.
\begin{aligned}[b]
  \int_{\Omega}F^N(\nabla u_{\epsilon, a})dx&=\int_{\Omega}F^N(\nabla u_{\epsilon})dx- \int_{\Omega}F^N(\nabla( u_{\epsilon}- u_{\epsilon, a}))dx\\
                                 &=1+ \gamma\| u_\epsilon\|_p^{N}-\int_{\Omega}F^N(\nabla( u_{\epsilon}- u_{\epsilon, a}))dx.
\end{aligned}
\right.
\end{equation*}
which implies
   $$\limsup\limits_{\epsilon \rightarrow 0}\int_{\Omega}F^N(\nabla u_{\epsilon, a})dx\leq a.$$
We have finished the proof of the lemma. 
\end{proof}
\begin{lemma} \label{lem4.4}
There holds
$$\lim_{\epsilon\rightarrow 0}\int_{\Omega}e^{(\lambda_{N}-\epsilon)
     |u_\epsilon|^{\frac{N}{N-1}}}dx\leq |\Omega|+\limsup_{\epsilon\rightarrow 0}\frac{\lambda_\epsilon}{c_\epsilon^{N/(N-1)}}.$$
\end{lemma}
\begin{proof}[\textbf{Proof}]
 For any $0<a< 1$ , by Lemma \ref{lem4.3}, we have  $\lim\limits_{\epsilon \rightarrow 0}\int_\Omega F^N(\nabla u_{\epsilon, a})dx=a<1$. From anisotropic Trudinger-Moser inequality,  $e^{\lambda_{N}u_{\epsilon, a}^{\frac{N}{N-1}}}$ is bounded in $L^q(\Omega)$ for some $q>1$. Notice that $u_{\epsilon, a}\rightarrow 0$ a.e. in $\Omega$, which implies $\lim_{\epsilon\rightarrow 0}\int_\Omega e^{\lambda_{N}u_{\epsilon, a}^{\frac{N}{N-1}}}=|\Omega|$. Hence
\begin{equation*}
\left.
\begin{aligned}[b]
e^{(\lambda_{N}-\epsilon)u_{\epsilon}^{\frac{N}{N-1}}}
= &\int\limits_{\{u_\epsilon\leq a c_\epsilon\}}e^{(\lambda_{N}-\epsilon)u_\epsilon^{\frac{N}{N-1}}} dx
+\int\limits_{ \{u_\epsilon> a c_\epsilon\} } e^{(\lambda_{N}-\epsilon)u_\epsilon^{\frac{N}{N-1}}} dx\\
\leq& \int\limits_{\{u_\epsilon\leq a c_\epsilon\}}e^{(\lambda_{N}-\epsilon)u_{\epsilon, a}^{\frac{N}{N-1}}} dx
+\frac{1}{(ac_\epsilon)^{\frac{N}{N-1}}}\int\limits_{ \{u_\epsilon> a c_\epsilon\} }u_{\epsilon}^{\frac{N}{N-1}} e^{(\lambda_{N}-\epsilon)u_\epsilon^{\frac{N}{N-1}}} dx\\
\leq& \int_{\Omega}e^{(\lambda_{N}-\epsilon)u_{\epsilon, a}^{\frac{N}{N-1}}} dx
+\frac{\lambda_\epsilon}{(ac_\epsilon)^{\frac{N}{N-1}}}.
\end{aligned}
\right.
\end{equation*}
Letting $\epsilon \rightarrow 0$ and $a \rightarrow 1$, we conclude our proof.  
\end{proof}

As  a consequence of Lemma \ref{lem4.4}, for any $\theta<\frac{N}{N-1}$, there holds
  \be \lim_{\epsilon\rightarrow 0}\frac{\lambda_\epsilon}{c_\epsilon^\theta}=+\infty.  \lb{4.10}\ee
If otherwise,  $\frac{\lambda_\epsilon}{c_\epsilon^{N/(N-1)}} \rightarrow 0$ as $\epsilon \rightarrow 0$, by  Lemma \ref{lem4.4}, we have
    $\lim\limits_{\epsilon\rightarrow 0}\int_{\Omega}e^{(\lambda_{N}-\epsilon)
     |u_\epsilon|^{\frac{N}{N-1}}}dx\leq |\Omega|$, which is impossible.
\begin{lemma} \label{lem4.5}
For any $1<q<N$, $c_\epsilon^{\frac{1}{N-1}}u_\epsilon \rightharpoonup G$ in $W^{1, q}_0(\Omega)$, where
$G$ is a distributional solution to
 \begin{equation}
   \left\{  \begin{array}{l}
          -\mathrm{div}(F^{N-1}(\nabla G)F_{\xi}(\nabla G))=\delta_{x_0}+\gamma \|G\|_p^{N-p}G^{p-1} ~~~~~~in~~~\Omega,\\
         G=0~~~~~~on~~~\partial\Omega.
         \end{array}
   \right.   \lb{4.9}
 \end{equation}
Furthermore, $c_\epsilon^{\frac{1}{N-1}}u_\epsilon \rightarrow G$ in $C^1_{loc}(\overline{\Omega}\backslash \{x_0\})$,  and $G$ has the form
  \be G(x)=-\frac{1}{(N\kappa_{N})^{\frac{1}{N-1}}}\log F^o(x-x_0)+A_{x_0}+\xi(x),    \lb{4.11}\ee
where $A_{x_0}$ is a constant depending only on $x_0$, $\xi \in C^0(\overline{\Omega})\bigcap C^1_{loc}(\Omega \backslash \{x_0\})$ and $\xi(x)=O(F^o(x-x_0))$ as $x\rightarrow x_0$.
\end{lemma}
\begin{proof}[\textbf{Proof}]
  Firstly, we claim  for any function $\phi\in C(\bar{\Omega})$, it holds
    \be  \lim\limits_{\epsilon\rightarrow 0}\int_{\Omega}\frac{c_\epsilon}{\lambda_\epsilon}u_\epsilon^{\frac{1}{N-1}}e^{(\lambda_{N}-\epsilon)u_\epsilon^{\frac{N}{N-1}}}\phi dx=\phi(x_0). \lb{4.12}\ee
In fact, for any $0<a< 1$ and $R > 0$, we have  $\mathcal{W}_{Rr_\epsilon}(x_\epsilon)\subset \{u_\epsilon > ac_\epsilon \}$ as $\epsilon > 0$ small enough,  we denote
 \begin{equation*}
\left.
\begin{aligned}[b]
&\int_{\Omega}\frac{c_\epsilon}{\lambda_\epsilon}u_\epsilon^{\frac{1}{N-1}}e^{(\lambda_{N}-\epsilon)u_\epsilon^{\frac{N}{N-1}}}\phi dx\\
=&\int\limits_{\{u_\epsilon\leq a c_\epsilon\}}\frac{c_\epsilon}{\lambda_\epsilon}u_\epsilon^{\frac{1}{N-1}}e^{(\lambda_{N}-\epsilon)u_\epsilon^{\frac{N}{N-1}}}\phi dx
+\int\limits_{ \mathcal{W}_{Rr_\epsilon}(x_\epsilon) } \frac{c_\epsilon}{\lambda_\epsilon}u_\epsilon^{\frac{1}{N-1}}e^{(\lambda_{N}-\epsilon)u_\epsilon^{\frac{N}{N-1}}}\phi dx\\
&+\int\limits_{\{u_\epsilon>a c_\epsilon\}\backslash \mathcal{W}_{Rr_\epsilon}(x_\epsilon) } \frac{c_\epsilon}{\lambda_\epsilon}u_\epsilon^{\frac{1}{N-1}}e^{(\lambda_{N}-\epsilon)u_\epsilon^{\frac{N}{N-1}}}\phi dx\\
:=&\uppercase\expandafter{\romannumeral 1}+\uppercase\expandafter{\romannumeral 2}+\uppercase\expandafter{\romannumeral 3}.
\end{aligned}
\right.
\end{equation*}
By Lemma \ref{lem4.3} and \eqref{4.9}, we have
 \begin{equation*}
\left.
\begin{aligned}[b]
\uppercase\expandafter{\romannumeral 1}=
                      &\frac{c_\epsilon}{\lambda_\epsilon}\int_{\{u_\epsilon\leq a c_\epsilon\}}u_\epsilon^{\frac{1}{N-1}}e^{(\lambda_{N}-\epsilon)u_\epsilon^{\frac{N}{N-1}}}\phi dx\\
                      &\leq\frac{c_\epsilon}{\lambda_\epsilon}\int_{\Omega}u_{\epsilon, a}^{\frac{1}{N-1}}e^{(\lambda_{N}-\epsilon)u_{\epsilon, a}^{\frac{N}{N-1}}}\phi dx\\
                      &=o_\epsilon(1)O(a).
\end{aligned}
\right.
\end{equation*}
Making the change in variable $x = x_\epsilon + r_\epsilon y$,  since $\phi(x_\epsilon + r_\epsilon \cdot) \rightarrow \phi(x_0)$ uniformly in  $\mathcal{W}_{R}(0)$. Note $ r_\epsilon^N=\lambda_\epsilon c_\epsilon^{-N/(N-1)}e^{-(\lambda_{N}-\epsilon)c_\epsilon^{N/(N-1)}}$, together with \eqref{4.7}, we have
 \begin{equation*}
\left.
\begin{aligned}[b]
\uppercase\expandafter{\romannumeral 2}=&
                     \int\limits_{ \mathcal{W}_{Rr_\epsilon}(x_\epsilon) } \frac{c_\epsilon}{\lambda_\epsilon}u_\epsilon^{\frac{1}{N-1}}e^{(\lambda_{N}-\epsilon)u_\epsilon^{\frac{N}{N-1}}}\phi dx\\
                     =&\int\limits_{ \mathcal{W}_{Rr_\epsilon}(x_\epsilon) } c_{\epsilon}r_\epsilon^{-N} c_\epsilon^{-N/(N-1)}e^{-(\lambda_{N}-\epsilon)c_\epsilon^{N/(N-1)}}u_\epsilon^{\frac{1}{N-1}}e^{(\lambda_{N}-\epsilon)u_\epsilon^{\frac{N}{N-1}}}\phi dx\\
                      =&\int\limits_{ \mathcal{W}_{R}(0) } c_{\epsilon}r_\epsilon^{-N} c_\epsilon^{-N/(N-1)}c_\epsilon^{1/(N-1)}\psi_\epsilon(y)^{\frac{1}{N-1}}e^{(\lambda_{N}-\epsilon)c_\epsilon^{\frac{N}{N-1}}(\psi_\epsilon(y)^{\frac{N}{N-1}}-1)}\phi(x_\epsilon + r_\epsilon y) r_\epsilon^Ndy\\
                     =&\int\limits_{ \mathcal{W}_{R}(0) } \psi_\epsilon(y)^{\frac{1}{N-1}}e^{(\lambda_{N}-\epsilon)c_\epsilon^{\frac{N}{N-1}}(\psi_\epsilon(y)^{\frac{N}{N-1}}-1)}\phi(x_\epsilon + r_\epsilon y) dy\\
                      \rightarrow  &\phi (x_0)
\end{aligned}
\right.
\end{equation*}
as $\epsilon\rightarrow 0$ and $R\rightarrow \infty$.
 \begin{equation*}
\left.
\begin{aligned}[b]
\uppercase\expandafter{\romannumeral 3}=&
\int\limits_{\{u_\epsilon>a c_\epsilon\}\backslash \mathcal{W}_{Rr_\epsilon}(x_\epsilon) } \frac{c_\epsilon}{\lambda_\epsilon}u_\epsilon^{\frac{1}{N-1}}e^{(\lambda_{N}-\epsilon)u_\epsilon^{\frac{N}{N-1}}}\phi dx\\
\leq& \|\phi\|_{\infty}\bigg(\int\limits_{\{u_\epsilon>a c_\epsilon\}} \frac{c_\epsilon}{\lambda_\epsilon}u_\epsilon^{\frac{1}{N-1}}e^{(\lambda_{N}-\epsilon)u_\epsilon^{\frac{N}{N-1}}} dx- \int\limits_{\mathcal{W}_{Rr_\epsilon}(x_\epsilon) } \frac{c_\epsilon}{\lambda_\epsilon}u_\epsilon^{\frac{1}{N-1}}e^{(\lambda_{N}-\epsilon)u_\epsilon^{\frac{N}{N-1}}} dx  \bigg)\\
\leq& \|\phi\|_{\infty}\bigg(\frac{1}{a}\int\limits_{\{u_\epsilon>a c_\epsilon\}} \frac{u_\epsilon^{\frac{N}{N-1}}}{\lambda_\epsilon}e^{(\lambda_{N}-\epsilon)u_\epsilon^{\frac{N}{N-1}}} dx- \int\limits_{\mathcal{W}_{Rr_\epsilon}(x_\epsilon) } \frac{c_\epsilon}{\lambda_\epsilon}u_\epsilon^{\frac{1}{N-1}}e^{(\lambda_{N}-\epsilon)u_\epsilon^{\frac{N}{N-1}}} dx  \bigg)\\
\leq& \|\phi\|_{\infty}\bigg(\frac{1}{a}- \int\limits_{\mathcal{W}_{Rr_\epsilon}(x_\epsilon) } \frac{c_\epsilon}{\lambda_\epsilon}u_\epsilon^{\frac{1}{N-1}}e^{(\lambda_{N}-\epsilon)u_\epsilon^{\frac{N}{N-1}}} dx  \bigg).
\end{aligned}
\right.
\end{equation*}
Thus
      $$\lim_{a \rightarrow 1}\lim_{R \rightarrow \infty}\lim_{\epsilon \rightarrow 0}\int\limits_{\{u_\epsilon>a c_\epsilon\}\backslash \mathcal{W}_{Rr_\epsilon}(x_\epsilon) } \frac{c_\epsilon}{\lambda_\epsilon}u_\epsilon^{\frac{1}{N-1}}e^{(\lambda_{N}-\epsilon)u_\epsilon^{\frac{N}{N-1}}}\phi dx=0. $$
Combing the above discussion, we have proved the claim. From \eqref{3.2}, we have
\begin{equation}\label{4.13}
                          -Q_{N}(c_\epsilon^{\frac{1}{N-1}}u_\epsilon)=\frac{c_\epsilon}{\lambda_\epsilon} u_\epsilon^{\frac{1}{N-1}}e^{(\lambda_{N}-\epsilon)u_\epsilon^{\frac{N}{N-1}}}+\gamma \|c_\epsilon^{\frac{1}{N-1}}u_\epsilon\|_p^{N-p}(c_\epsilon^{\frac{1}{N-1}}u_\epsilon)^{p-1}.
\end{equation}
It follows from \eqref{4.12} that $\frac{c_\epsilon}{\lambda_\epsilon}u_\epsilon^{\frac{1}{N-1}}e^{(\lambda_{N}-\epsilon)u_\epsilon^{\frac{N}{N-1}}}$ is bounded in $L^1(\Omega)$. From Lemma \ref{lem2.2}, we know that $c_\epsilon^{\frac{1}{N-1}}u_\epsilon$ is  bounded in $W_0^{1, q}(\Omega)$ for any $1<q<N$. Hence there exists some $G\in W_0^{1, q}(\Omega)$ such that $c_\epsilon^{\frac{1}{N-1}}u_\epsilon\rightharpoonup G$ in $W_0^{1, q}(\Omega)$ for any $1<q<N$. Multiplying \eqref{4.13} with $\phi\in C_0^\infty(\Omega)$ , we get
\begin{equation*}
\left.
\begin{aligned}[b]
  -\int_\Omega \phi Q_N(c_\epsilon^{\frac{1}{N-1}}u_\epsilon)dx=& \int_\Omega \phi\frac{c_\epsilon}{\lambda_\epsilon} u_\epsilon^{\frac{1}{N-1}}e^{(\lambda_{N}-\epsilon)u_\epsilon^{\frac{N}{N-1}}}dx\\
     +&\gamma \|c_\epsilon^{\frac{1}{N-1}}u_\epsilon\|_p^{N-p}\int_\Omega \phi(c_\epsilon^{\frac{1}{N-1}}u_\epsilon)^{p-1}dx.
\end{aligned}
\right.
\end{equation*}
Let $\epsilon\rightarrow 0$, using again \eqref{4.12}, we have
$$\int_\Omega \nabla\phi F^{N-1}(\nabla G)F_{\xi}(\nabla G)dx=\phi (x_0)+\gamma \|G\|_p^{N-p}\int_\Omega \phi G^{p-1}dx. $$
Therefore, in the distributional sense,
  $$  -\mathrm{div}(F^{N-1}(\nabla G)F_{\xi}(\nabla G))=\delta_{x_0}+\gamma\|G\|_p^{N-p}G^{p-1} ~~~ \mathrm{in} ~~~ \Omega. $$
 Applying with the standard elliptic regularity theory in \cite{P.Tolksdorf},  we get $c_\epsilon^{\frac{1}{N-1}}u_\epsilon \rightarrow G$ in $C^1_{loc}(\overline{\Omega}\backslash \{x_0\})$.   Finally, replacing the right hand term in equation \eqref{4.11}, we use the similar discussion as Lemma 4.7 in \cite{Zhou}, the asymptotic representation of Green function can immediateiy derived. This complete the proof of Lemma \ref{lem4.5}. 
 \end{proof}

Next, we will exclude the case $x_0\in \partial\Omega$. Denote $d_\epsilon=dist(x_\epsilon, \partial\Omega)$ and $r_\epsilon$ be defined by \eqref{4.2}. We obtain
\begin{lemma} \label{lem4.6} If $0\leq\gamma <\gamma_1$ and $x_0\in \partial\Omega$. Then $\frac{r_\epsilon}{d_\epsilon}\rightarrow 0$ as $\epsilon \rightarrow 0$. Moreover, 
 $c_\epsilon^{\frac{1}{N-1}}u_\epsilon\rightharpoonup0$  weakly in $W^{1, q}_0(\Omega)(1<q<N)$  and
$c_\epsilon^{\frac{1}{N-1}}u_\epsilon\rightarrow 0$ strongly in $C^1(\overline{\Omega}\backslash \{x_0\}).$
\end{lemma}
\begin{proof}[\textbf{Proof}]
Firstly, we prove that $\frac{r_\epsilon}{d_\epsilon}\rightarrow 0$. If not, there exist some constant $\delta$ such that $\frac{r_\epsilon}{d_\epsilon} \geq \delta$ and $y_\epsilon \in \partial\Omega$, $d_\epsilon=|x_\epsilon-y_\epsilon|$. Let
$$\tilde{\psi}_\epsilon=\frac{u_\epsilon(y_\epsilon+r_\epsilon x)}{c_\epsilon}.$$ 
As the similar procedure of interior case, we have $$\tilde{\psi}_\epsilon \rightarrow 1 \quad \text{in}   \quad  C^1(B_R^+) \quad  \text{for} \quad  \|\tilde{\psi}_\epsilon\|_{L^\infty(\overline{B_R^+})}=1.$$ 
This is impossible because of $\tilde{\psi}_\epsilon(0)=0$. Secondly, let $\Omega_\epsilon=\{x\in \mathbb{R}^N: x_\epsilon+r_\epsilon\epsilon \in \Omega\}$, we have know that  $\frac{r_\epsilon}{d_\epsilon}\rightarrow 0$ by the first step, then $\Omega _\epsilon \rightarrow \mathbb{R}^N$. Let $\varphi_\epsilon$ and $\varphi$,   the  same argument
as the proof of Lemma \ref{lem4.2},  we get $\varphi_\epsilon \rightarrow \varphi$ in $C_{loc}^1(\mathbb{R}^N)$. By the similar process as interior case,  we have 	$c_\epsilon^{\frac{1}{N-1}}u_\epsilon\rightharpoonup \tilde{G}$ in   weakly in $W^{1, q}_0(\Omega)(1<q<N)$  and
 in $C^1(\overline{\Omega}\backslash \{x_0\})$ with $\tilde{G}$ satisfying  $-Q_N \tilde{G}=\gamma\|\tilde{G}\|_p^{N-p}\tilde{G}^{p-1}$ in $\Omega$ and $\tilde{G}=0$ on $\partial \Omega$. By the standard elliptic regularity theory, we have $\tilde{G}\in C^1(\overline{\Omega})$. Since $\gamma <\gamma_1$,  test the eqution with function $\tilde{G}$, we get $\tilde{G}\equiv 0$. Thus we have
  $c_\epsilon^{\frac{1}{N-1}}u_\epsilon\rightharpoonup0$  weakly in $W^{1, q}_0(\Omega)(1<q<N)$  and
 $c_\epsilon^{\frac{1}{N-1}}u_\epsilon\rightarrow 0$ strongly in $C^1(\overline{\Omega}\backslash \{x_0\}).$ 
\end{proof}

\begin{lemma} \label{lem4.7}
If $0< \gamma <\gamma_1$, the blow-up point $x_0 \notin \partial\Omega$.
\end{lemma}
\begin{proof}[\textbf{Proof}]
 Suppose $x_0 \in \partial\Omega$. Then $\|u_\epsilon\|_N^N \rightarrow 0$ implies $$(1+\gamma\|u_\epsilon\|_p^N)^{-\frac{1}{N-1}}=1-\frac{\gamma}{N-1}\|u_\epsilon\|_p^N+O(\|u_\epsilon\|_p^{2N}).$$
Let $w_\epsilon=\frac{u_\epsilon}{\int_{\Omega}F^{N}(\nabla u_\epsilon)}$, since $\int_{\Omega}F^{N}(\nabla u_\epsilon)=1+\gamma\|u_\epsilon\|_p^N$, we have
\begin{equation*}
\left.
\begin{aligned}  
 \Lambda_{\gamma, \epsilon} =&\int_{\Omega}e^{(\lambda_{N}-\epsilon)
	|u_{\epsilon}|^{\frac{N}{N-1}}}dx=\int_{\Omega}e^{(\lambda_{N}-\epsilon)
	(1+\gamma\|u_\epsilon\|_p^N)^{\frac{1}{N-1}}|w_{\epsilon}|^{\frac{N}{N-1}}}dx\\
   =&\int_{\Omega}e^{[(\lambda_{N}-\epsilon)
   	(1+\gamma\|u_\epsilon\|_p^N)^{\frac{1}{N-1}}-1]w_\epsilon^{\frac{N}{N-1}}}e^{(\lambda_{N}-\epsilon)
   	|w_{\epsilon}|^{\frac{N}{N-1}}}dx \\
   \leq&\int_{\Omega}e^{\lambda_{N}
   	[1-(1+\gamma\|u_\epsilon\|_p^N)^{-\frac{1}{N-1}}]c_{ \epsilon}^{\frac{N}{N-1}}}\int_{\Omega}e^{\lambda_{N}
   	|w_{\epsilon}|^{\frac{N}{N-1}}}dx\\
     \leq&e^{\lambda_{N}
    [\frac{\gamma}{N-1}\|c_{ \epsilon}^{\frac{1}{N-1}}u_\epsilon\|_p^N+c_{ \epsilon}^{-\frac{N}{N-1}}O(\|c_{ \epsilon}^{\frac{1}{N-1}}u_\epsilon\|_p^N)]}\Lambda_{0}.
\end{aligned}
\right. 
\end{equation*}
 From  Lemma \ref{lem4.6}, we have $\|c_\epsilon^{\frac{1}{N-1}}u_\epsilon\|_p^N\rightarrow0$. Thus
lettint $\epsilon\rightarrow 0$ and using \eqref{3.2}, we have $ \Lambda_{\gamma} \leq \Lambda_{0}$.

On the other hand, according to the anisotropic Trudinger-Moser inequality, $\Lambda_{0}$  is attained by a function $u\in W^{1, N}_0(\Omega)$ with $\int_{\Omega}F^{N}(\nabla u)=1$. Define $v=u/({1-\gamma\int_{\Omega}F^{N}(\nabla u)})^{\frac{1}{N}}$. Thus $$\|v\|_{N,F,\gamma, p}=\bigg(\int_{\Omega}F^{N}(\nabla v)dx-\gamma\| v\|_p^{N}\bigg)^{\frac{1}{N}}=1$$
Since $u\not\equiv 0$ and $\gamma>0$, we get $|u|\leq|v|$ and $|u|\not\equiv|v|$. Thus
\begin{equation*}
  \Lambda_{\gamma}\geq \int_{\Omega}e^{\lambda_{N}
  	|v|^{\frac{N}{N-1}}}dx> \int_{\Omega}e^{\lambda_{N}
  	|u|^{\frac{N}{N-1}}}dx=\Lambda_{0},
\end{equation*}
which is a contradiction with $ \Lambda_{\gamma} \leq \Lambda_{0}$.
\end{proof}
\subsection{The upper bound estimate}
  We will use the capacity technique to give an upper bound  estimate,  which was  used by Y. Li \cite{Y. Li} and Yang-Zhu \cite{YZ}, our main result of this subsection is an upper bound estimate.
\begin{lemma} \label{lem4.8}
$\Lambda_{\gamma}\leq|\Omega|+\kappa_{N}e^{\lambda_NA_{x_0}+\sum_{k=1}^{N-1}\frac{1}{k}}.$
\end{lemma}
\begin{proof}[\textbf{Proof}]
 Notice that $x_0\in \Omega$. Take $\delta > 0$ such that $\mathcal{W}_{\delta}(x_0)\subset \Omega$. For any $R>0$, we assume that $\epsilon$ is so small that $\delta > Rr_\epsilon$. We denote by $o_\epsilon(1)$ ($o_\delta(1)$; $o_R(1)$) the terms which tend to $0$ as 
$\epsilon\rightarrow 0$ ($\delta\rightarrow 0$; $R\rightarrow \infty$). 
From Lemma \ref{lem4.5}, we have
\begin{equation}\label{4.14}
\left.
\begin{aligned}
\int\limits_{\Omega \backslash \mathcal{W}_{\delta}(x_\epsilon) }F^N(\nabla u_\epsilon)dx=&\frac{1}{c_\epsilon^{\frac{N}{N-1}}}(\int_{\Omega \backslash \mathcal{W}_{\delta}(x_\epsilon) }F^N(\nabla G)dx+o_\epsilon(1))\\
=&\frac{1}{c_\epsilon^{\frac{N}{N-1}}}(\int_{\Omega \backslash \mathcal{W}_{\delta}(x_\epsilon) }-\mathrm{div} (F^{N-1}(\nabla G)F_\xi(\nabla G)) G dx\\
&+\int_{\partial(\Omega \backslash \mathcal{W}_{\delta}(x_\epsilon)) } GF^{N-1}(\nabla G)\langle F_\xi(\nabla G), \nu\rangle dx+o_\epsilon(1))\\
=&\frac{1}{c_\epsilon^{\frac{N}{N-1}}}(\gamma\int_{\Omega \backslash \mathcal{W}_{\delta}(x_\epsilon) }\|G\|_p^{N-p}G^{p}dx\\
&+\int_{\partial(\Omega \backslash \mathcal{W}_{\delta}(x_\epsilon)) }G F^{N-1}(\nabla G)\langle F_\xi(\nabla G), \nu\rangle  dx+o_\epsilon(1))\\
=&\frac{1}{c_\epsilon^{\frac{N}{N-1}}}(\gamma\|G\|_p^N-\int_{\partial \mathcal{W}_{\delta}(x_\epsilon) }G F^{N-1}(\nabla G)\langle F_\xi(\nabla G), \nu\rangle  dx+o_\delta(1)+o_\epsilon(1))\\
=&\frac{1}{c_\epsilon^{\frac{N}{N-1}}}(\gamma\|G\|_p^N-\frac{1}{(N\kappa_N)^{\frac{1}{N-1}}}\log \delta+A_{x_0}+o_\delta(1)+o_\epsilon(1))\\
=&\frac{1}{c_\epsilon^{\frac{N}{N-1}}}(\gamma\|G\|_p^N-\frac{N}{\lambda_N}\log \delta+A_{x_0}+o_\delta(1)+o_\epsilon(1)).
\end{aligned}
\right.  
\end{equation}
From  \eqref{4.4}, we have on $\mathcal{W}_{Rr_\epsilon}(x_\epsilon)$  that $u _\epsilon(x) = c_\epsilon^{-\frac{1}{N-1}} \varphi_\epsilon(\frac{x-x_\epsilon}{r_\epsilon} ) +c_\epsilon$. Thus
\begin{equation}\label{4.15}
\int\limits_{\mathcal{W}_{Rr_\epsilon}(x_\epsilon) }F^N(\nabla u_\epsilon)dx=c_\epsilon^{-\frac{N}{N-1}}\int\limits_{\mathcal{W}_{R}(0) }F^N(\nabla \varphi_\epsilon)dx
=\frac{1}{c_\epsilon^{\frac{N}{N-1}}}\bigg(\int\limits_{\mathcal{W}_{R}(x_0) }F^N(\nabla \varphi)dx+o_\epsilon(1)\bigg).
\end{equation}
By a straightforward computation, we have
\begin{equation}\label{4.16}
\left.
\begin{aligned}
\int\limits_{\mathcal{W}_{R}(x_0) }F^N(\nabla \varphi)dx
=\frac{N}{\lambda_{N}}\log R+\frac{1}{\lambda_{N}}\log \kappa_N
-\frac{N-1}{\lambda_N }\sum_{k=1}^{N-1}\frac{1}{k}+o_R(1).
\end{aligned}
\right.  
\end{equation}
Let $i_\epsilon=\inf_{\partial \mathcal{W}_{Rr_\epsilon}(x_\epsilon)}u_\epsilon$, $s_\epsilon=\sup_{\partial \mathcal{W}_{\delta}(x_\epsilon)}u_\epsilon$.  Since $ \psi_{ \epsilon}(x)\rightarrow 1$ in $C_{loc}^1( \mathbb{R}^{N})$, together with Lemma \ref{lem4.5}, we have $i_\epsilon >s_\epsilon$ for $\epsilon> 0$ small enough. Define a function space 
 $$\mathbf{S}(i_\epsilon, s_\epsilon)=\{u \in \mathcal{W}_{\delta}(x_\epsilon) \backslash \overline{\mathcal{W}_{Rr_\epsilon}(x_\epsilon)}: u|_{\partial \mathcal{W}_{Rr_\epsilon}}= i_\epsilon,  u|_{\partial \mathcal{W}_{\delta}(x_\epsilon)}= s_\epsilon\},$$
and $\inf\limits_{u\in \mathbf{S}(i_\epsilon, s_\epsilon)} \int_{ B_{\delta}(x_\epsilon) \backslash B_{Rr_\epsilon}(x_\epsilon)} F^N(\nabla u)dx$ is attained by $h(x)$ satisfying
 \begin{equation}\label{4.17}
\begin{cases}
-Q_N h=0  &  \text{in}\quad \mathcal{W}_{\delta}(x_\epsilon) \backslash \overline{\mathcal{W}_{Rr_\epsilon}(x_\epsilon)},\\
h|_{\partial \mathcal{W}_{Rr_\epsilon}} = i_\epsilon,\\
 h|_{\partial \mathcal{W}_{\delta}(x_\epsilon)} = s_\epsilon.
\end{cases}
\end{equation}
The unique solution is
   \begin{equation}\label{4.18}
   h(x)=\frac{s_\epsilon (\log F^0(x-x_\epsilon)-\log(R r_\epsilon))+i_\epsilon (-\log\delta-\log F^0(x-x_\epsilon))}{\log \delta-\log(R r_\epsilon)},
      \end{equation}
 and hence 
  \begin{equation}\label{4.19}
  \int\limits_{ \mathcal{W}_{\delta}(x_\epsilon) \backslash \mathcal{W}_{Rr_\epsilon}(x_\epsilon)} F^N(\nabla h)dx=N\kappa_N\frac{(i_\epsilon-s_\epsilon)^N}{(\log \delta-\log(R r_\epsilon))^{N-1}}.
      \end{equation}
 Define $\tilde{u}_\epsilon=\max \{s_\epsilon, \min\{u_\epsilon, i_\epsilon\}\}$. Then $\tilde{u}_\epsilon \in  \mathbf{S}(i_\epsilon, s_\epsilon)$ and $F(\tilde{u}_\epsilon)\leq F(u_\epsilon)$ for $\epsilon>0$ small enough. Therefore
   \begin{equation}\label{4.20}
   \left.
   \begin{aligned}
   \int\limits_{ \mathcal{W}_{\delta}(x_\epsilon) \backslash \mathcal{W}_{Rr_\epsilon}(x_\epsilon)} F^N(\nabla h)dx=&\int\limits_{\mathcal{W}_{\delta}(x_\epsilon) \backslash \mathcal{W}_{Rr_\epsilon}(x_\epsilon) }F^N(\nabla \tilde{u}_\epsilon)dx\\
   \leq&\int\limits_{\mathcal{W}_{\delta}(x_\epsilon) \backslash \mathcal{W}_{Rr_\epsilon}(x_\epsilon) }F^N(\nabla u_\epsilon)dx\\
 =&1+\gamma \|u_\epsilon\|_p^N-\int\limits_{\mathcal{W}_{Rr_\epsilon}(x_\epsilon) }F^N(\nabla u_\epsilon)dx-\int\limits_{\Omega \backslash \mathcal{W}_{\delta}(x_\epsilon) }F^N(\nabla u_\epsilon)dx.
   \end{aligned}
   \right.  
   \end{equation}
Since $\gamma \|u_\epsilon\|_p^N=\gamma c_\epsilon^{-\frac{N}{N-1}} ( \|G\|_p^N+o_\epsilon(1))$, combing \eqref{4.14}-\eqref{4.16} and \eqref{4.19}-\eqref{4.20}, we obtain 
\begin{equation}\label{4.21}
\left.
\begin{aligned}
&\frac{\lambda_N}{N}\frac{(i_\epsilon-s_\epsilon)^\frac{N}{N-1}}{\log \frac{\delta}{R}-\log r_\epsilon} \\
 \leq & \bigg(1+\frac{\frac{N}{\lambda_N}\log \frac{\delta}{R}-\frac{1}{\lambda_{N}}\log \kappa_N+\frac{N-1}{\lambda_N }\sum_{k=1}^{N-1}\frac{1}{k}-A_{x_0}+o_\delta(1)+o_\epsilon(1)
	+o_R(1)}{c_\epsilon^{\frac{N}{N-1}}}\bigg)^{\frac{1}{N-1}}\\
\leq & 1+\frac{\frac{N}{\lambda_N}\log \frac{\delta}{R}-\frac{1}{\lambda_{N}}\log \kappa_N+\frac{N-1}{\lambda_N }\sum_{k=1}^{N-1}\frac{1}{k}-A_{x_0}+o_\delta(1)+o_\epsilon(1)
	+o_R(1)}{(N-1)c_\epsilon^{\frac{N}{N-1}}}
\end{aligned}
\right.  
\end{equation}
here we use the inequality $(1+t)^{\frac{1}{N-1}} \leq 1+ \frac{1}{N-1}t$
for any $-1 \leq t \leq 0$ with
   $$-1\leq \frac{\frac{N}{\lambda_N}\log \frac{\delta}{R}-\frac{1}{\lambda_{N}}\log \kappa_N+\frac{N-1}{\lambda_N }\sum_{k=1}^{N-1}\frac{1}{k}-A_{x_0}+o_\delta(1)+o_\epsilon(1)
   	+o_R(1)}{c_\epsilon^{\frac{N}{N-1}}}\leq 0.$$
 On the other hand, using \eqref{4.8} and Lemma \ref{lem4.5}, we have
  \begin{equation}\label{4.22}
  \left.
  \begin{aligned}
  &(i_\epsilon-s_\epsilon)^{\frac{N}{N-1}}\\
  = & c_\epsilon^{\frac{N}{N-1}} \bigg(1+\frac{\frac{N}{\lambda_N}\log \frac{\delta}{R}-\frac{1}{\lambda_{N}}\log \kappa_N-A_{x_0}+o_\delta(1)+o_\epsilon(1)
  	+o_R(1)}{c_\epsilon^{\frac{N}{N-1}}}\bigg)^{\frac{N}{N-1}}\\
 \geq & c_\epsilon^{\frac{N}{N-1}}+\frac{N}{N-1}\bigg(\frac{N}{\lambda_N}\log \frac{\delta}{R}-\frac{1}{\lambda_{N}}\log \kappa_N-A_{x_0}+o_\delta(1)+o_\epsilon(1))
  	+o_R(1)\bigg),
  \end{aligned}
  \right.  
  \end{equation}
here we use the inequality $(1+t)^{\frac{N}{N-1}} \geq 1+ \frac{N}{N-1}t$
for any $-1 \leq t \leq 0$ with
$$-1\leq \frac{\frac{N}{\lambda_N}\log \frac{\delta}{R}-\frac{1}{\lambda_{N}}\log \kappa_N-A_{x_0}+o_\delta(1)+o_\epsilon(1)
	+o_R(1)}{c_\epsilon^{\frac{N}{N-1}}}\leq 0.$$ 
 Since $\log \frac{\delta}{R}-\log r_\epsilon=\log \frac{\delta}{R}+\frac{\lambda_N-\epsilon}{N}c_\epsilon^{\frac{N}{N-1}}-\frac{1}{N}\log\frac{\lambda_\epsilon}{c_\epsilon^{\frac{N}{N-1}}}$, combing \eqref{4.21} and \eqref{4.22}, we have
  \begin{equation}\label{4.23}
 \left.
 \begin{aligned}
 &\frac{\lambda_N}{N}\bigg[c_\epsilon^{\frac{N}{N-1}}+\frac{N}{N-1}\bigg(\frac{N}{\lambda_N}\log \frac{\delta}{R}-\frac{1}{\lambda_{N}}\log \kappa_N-A_{x_0}+o_\delta(1)+o_\epsilon(1))
 +o_R(1)\bigg)\bigg]\\
\leq &(\log \frac{\delta}{R}+\frac{\lambda_N-\epsilon}{N}c_\epsilon^{\frac{N}{N-1}}-\frac{1}{N}\log\frac{\lambda_\epsilon}{c_\epsilon^{\frac{N}{N-1}}})\\
&\times \bigg[1+\frac{\frac{N}{\lambda_N}\log \frac{\delta}{R}-\frac{1}{\lambda_{N}}\log \kappa_N+\frac{N-1}{\lambda_N }\sum_{k=1}^{N-1}\frac{1}{k}-A_{x_0}+o(1)}{(N-1)c_\epsilon^{\frac{N}{N-1}}}\bigg]\\
\leq &\frac{\lambda_N-\epsilon}{N}c_\epsilon^{\frac{N}{N-1}}    +\frac{N}{N-1}\log \frac{\delta}{R}-\frac{1+o(1)}{N}\log\frac{\lambda_\epsilon}{c_\epsilon^{\frac{N}{N-1}}}-\frac{1}{N(N-1)}\log \kappa_N\\
&+\frac{1}{N }\sum_{k=1}^{N-1}\frac{1}{k}-\frac{\lambda_N}{N(N-1)}A_{x_0}+o(1).
 \end{aligned}
 \right.  
 \end{equation} 
Thus
$$\frac{1+o(1)}{N}\log\frac{\lambda_\epsilon}{c_\epsilon^{\frac{N}{N-1}}}\leq \frac{1}{N}\log \kappa_N+\sum_{k=1}^{N-1}\frac{1}{k}+\frac{\lambda_N}{N}A_{x_0}++o_\delta(1)+o_\epsilon(1))
+o_R(1),$$  
which lead to 
   $$\limsup_{\epsilon \rightarrow 0}\frac{\lambda_\epsilon}{c_\epsilon^{\frac{N}{N-1}}}\leq \kappa_{N}e^{\lambda_NA_{x_0}+\sum_{k=1}^{N-1}\frac{1}{k}}.$$
Recall Lemma \ref{lem4.4}, we have finished  the proof.
\end{proof}

\section{Proof of main Theorems }\label{section 5}
\begin{proof}[\textbf{Proof of Theorem 1.1.}]
Let $0\leq\gamma <\gamma_1$. If $c_\epsilon$ is bounded,   the inequality \eqref{4.1} implies that the Theorem holds. If $c_\epsilon\rightarrow +\infty$   is bounded, we know that the blow-up point $x_0 \in \Omega$ by Lemma \ref{lem4.7} and the rusult  is followed from Lemma \ref{lem4.8}.
\end{proof} 
\begin{proof}[\textbf{Proof of Theorem 1.2.}]
 Let $0\leq\gamma <\gamma_1$. We prove that the blow-up phenomena do not occur. In fact, if  $c_\epsilon \rightarrow +\infty$. In Lemma  \ref{lem4.8}, we have got the upper bound of $\Lambda_{\gamma}$, that is to say
 \be  \Lambda_{\gamma}\leq|\Omega|+\kappa_{N}e^{\lambda_NA_{x_0}+\sum_{k=1}^{N-1}\frac{1}{k}},   \label{5.1}\ee
 where $x_0$ is the blow-up point. 
If we can construct a sequence $\phi_\epsilon \in W^{1, N}(\Omega)$ with $\|\phi_\epsilon\|^N_{N,F,\gamma, p}=\int_{\Omega}F^{N}(\nabla \phi_\epsilon)dx-\gamma\| \phi_\epsilon\|_p^N=1,$ but
  \be \int_{\Omega}e^{\lambda_{N} |\phi_\epsilon|^{\frac{N}{N-1}}}dx >|\Omega|+\kappa_{N}e^{\lambda_NA_{x_0}+\sum_{k=1}^{N-1}\frac{1}{k}}. \label{5.2}  \ee
This is the contradiction with \eqref{5.1}, which implies that $c_\epsilon$ must be bounded and can be attained  by the discussion at the beginning of subsection \ref{subsection 4.1}.  Thus it suffice to 
construct the sequence $\phi_\epsilon$ such that \eqref{5.2} holds when the blow-up phenomena occur. From Lemma \ref{lem4.5}, 
$$ G(x)=-\frac{1}{(N\kappa_{N})^{\frac{1}{N-1}}}\log F^o( x-x_0)+A_{x_0}+\xi(x).    $$
 Define a sequence of functions 
\begin{equation}\label{4.4}
\phi_\epsilon(x)=
\begin{cases}
	C+C^{-\frac{1}{N-1}}(-\frac{N-1}{\lambda_N}\log(1+\kappa_{N}^{\frac{1}{N-1}}(F^o( x-x_0)\epsilon^{-1})^{\frac{N}{N-1}}+B),  &  x\in \overline{\mathcal{W}_{R \epsilon}(x_0)},\\
C^{-\frac{1}{N-1}}(G-\eta \xi),  &x\in \mathcal{W}_{2R \epsilon}(x_0)\backslash \overline{\mathcal{W}_{R \epsilon}(x_0)},\\
C^{-\frac{1}{N-1}}G,  &  \Omega \backslash  \mathcal{W}_{2R \epsilon}(x_0),
\end{cases}   
\end{equation}
where $B$ and  $C$ are constants depending only on $\epsilon$, which will be determined later. The cutoff function $\eta \in C_0^{1}(\mathcal{W}_{2R \epsilon}(x_0))$, $0\leq \eta \leq1$ and $\eta=1$ in  $\mathcal{W}_{R \epsilon}(x_0)$. To ensure $\phi_\epsilon \in W^{1, N}(\Omega)$, we require for all $x \in \partial \mathcal{W}_{R \epsilon}(x_0)$, there holds 
  $$C+C^{-\frac{1}{N-1}}(-\frac{N-1}{\lambda_N}\log(1+\kappa_{N}^{\frac{1}{N-1}}R^{\frac{N}{N-1}}+B) =C^{-\frac{1}{N-1}}(-\frac{1}{(N\kappa_{N})^{\frac{1}{N-1}}}\log (R\epsilon)+A_{x_0}),$$
which implies  
$$B=-C^{\frac{N}{N-1}}+\frac{N-1}{\lambda_N}\log(1+\kappa_{N}^{\frac{1}{N-1}}R^{\frac{N}{N-1}}) -\frac{N}{\lambda_{N}}\log (R\epsilon)+A_{x_0}.$$
On one hand, since 
$$F^N(\nabla \phi_\epsilon)=C^{-\frac{N}{N-1}} F^N(\nabla \phi_\epsilon)(1+O(R\epsilon)) $$ uniformly in $\mathcal{W}_{2R}(x_0) \backslash \mathcal{W}_{R}(x_0)$ as $\epsilon \rightarrow 0$. Thus, by Lemma \ref{lem4.5}, we have
\begin{equation*}
\left.
\begin{aligned}
\int\limits_{\Omega\backslash \mathcal{W}_{R}(x_0)}F^N(\nabla \phi_\epsilon)dx=&\int\limits_{\Omega \backslash \mathcal{W}_{2R}(x_0)}F^N(\nabla \phi_\epsilon)dx+\int\limits_{ \mathcal{W}_{2R}(x_0) \backslash \mathcal{W}_{R}(x_0)}F^N(\nabla \phi_\epsilon)dx\\
=&C^{-\frac{N}{N-1}}(\int\limits_{\Omega \backslash \mathcal{W}_{2R}(x_0)}F^N(\nabla G)dx+\int\limits_{ \mathcal{W}_{2R}(x_0) \backslash \mathcal{W}_{R}(x_0)}F^N(\nabla G)(1+O(R\epsilon))dx)\\
=&C^{-\frac{N}{N-1}}(\int\limits_{\Omega \backslash \mathcal{W}_{R}(x_0)}F^N(\nabla G)dx+\int\limits_{ \mathcal{W}_{2R}(x_0) \backslash \mathcal{W}_{R}(x_0)}F^N(\nabla G)O(R\epsilon)dx)\\
=&C^{-\frac{N}{N-1}}(\int_{\Omega\backslash \mathcal{W}_{R}(x_0)}F^N(\nabla G)dx+O(-R\epsilon\log(R\epsilon)))\\
=&C^{-\frac{N}{N-1}}\bigg(\gamma\|G\|_p^N-\frac{N}{\lambda_{N}}\log(R\epsilon)+A_{x_0}+O(-R\epsilon\log(R\epsilon))\bigg).
\end{aligned}
\right.  
\end{equation*}
On the other hand, through the direct calculation, we have
\begin{equation*}
\left.
\begin{aligned}
\int\limits_{\mathcal{W}_{R}(x_0) }F^N(\nabla \phi_\epsilon)dx
=C^{-\frac{N}{N-1}}\bigg(\frac{N-1}{\lambda_{N}}\log(1+\kappa_N^\frac{1}{N-1} R^\frac{N}{N-1})-\frac{N-1}{\lambda_N }\sum_{k=1}^{N-1}\frac{1}{k}+o_R(1)\bigg).
\end{aligned}
\right.  
\end{equation*}
Thus   
\begin{equation*}
\left.
\begin{aligned}
\int_{\Omega}F^{N}(\nabla \phi_\epsilon)dx=&\int\limits_{\Omega\backslash \mathcal{W}_{R}(x_0)}+\int\limits_{\mathcal{W}_{R}(x_0) }F^N(\nabla \phi_\epsilon)dx\\
=&C^{-\frac{N}{N-1}}\bigg(\gamma\|G\|_p^N-\frac{N}{\lambda_{N}}\log\epsilon+A_{x_0}+\frac{1}{\lambda_{N}}\log \kappa_{N}\\
&-\frac{N-1}{\lambda_{N}}\sum_{k=1}^{N-1}\frac{1}{k}+O(-R\epsilon\log(R\epsilon))\bigg),
\end{aligned}
\right.  
\end{equation*}
we also have
\begin{equation*}
\left.
\begin{aligned}
\|\phi_\epsilon\|_p^N=C^{-\frac{N}{N-1}}\bigg(\|G\|_p^N+O((R\epsilon)^N(-\log(R\epsilon))^N)\bigg).
\end{aligned}
\right.  
\end{equation*}
Take $R=-\log \epsilon$, we get
\begin{equation}\label{5.4}
\left.
\begin{aligned}
 \|\phi_\epsilon\|^N_{N,F,\gamma, p}=&\int_{\Omega}F^{N}(\nabla \phi_\epsilon)dx-\gamma\| \phi_\epsilon\|_p^N\\
  =&C^{-\frac{N}{N-1}}\bigg(\frac{N}{\lambda_{N}}R+A_{x_0}+\frac{1}{\lambda_{N}}\log \kappa_{N}-\frac{N-1}{\lambda_{N}}\sum_{k=1}^{N-1}\frac{1}{k}+O(R^{-\frac{N}{N-1}})   \bigg)
\end{aligned}
\right.  
\end{equation}
Choosing 
\be C^{\frac{N}{N-1}}=-\frac{N}{\lambda_{N}}\log \epsilon+A_{x_0}+\frac{1}{\lambda_{N}}\log \kappa_{N}-\frac{N-1}{\lambda_{N}}\sum_{k=1}^{N-1}\frac{1}{k}+O(R^{-\frac{N}{N-1}}),  \label{5.5} \ee
which implies 
\be B=\frac{N-1}{\lambda_{N}}\sum_{k=1}^{N-1}\frac{1}{k}+O(R^{-\frac{N}{N-1}})+o_R(1).\label{5.6}    \ee
Then we can get  $\|\phi_\epsilon\|^N_{N,F,\gamma, p}=\int_{\Omega}F^{N}(\nabla \phi_\epsilon)dx-\gamma\| \phi_\epsilon\|_p^N=1.$ 

Now we estimate $\int_{\Omega} e^{\lambda_N|\phi_\epsilon|^{\frac{N}{N-1}}}$. Let $F^o(x-x_0)=\epsilon y$, we get
\begin{equation}\label{5.7}
\left.
\begin{aligned}[b]  
\int_{\mathcal{W}_{R\epsilon}(x_0)} e^{\lambda_N|\phi_\epsilon|^{\frac{N}{N-1}}}
\geq&\kappa_Ne^{\lambda_N A_{x_0}+\sum_{k=1}^{N}\frac{1}{k}+O(R^{-\frac{N}{N-1}})}\int\limits_{\mathcal{W}_{R}(0)}( 1+\kappa_N^{-\frac{1}{N-1}} |y|^{\frac{N}{N-1}} )^{-N}dy\\
=&\kappa_Ne^{\lambda_N A_{x_0}+\sum_{k=1}^{N}\frac{1}{k}+O(R^{-\frac{N}{N-1}})}(1+O(R^{-\frac{N}{N-1}}))\\
=&\kappa_Ne^{\lambda_N A_{x_0}+\sum_{k=1}^{N}\frac{1}{k}}+O(R^{-\frac{N}{N-1}}).
\end{aligned}
\right.   
\end{equation}
From the inequality $e^t \geq 1 + \frac{t^{N-1}}{(N-1)!}$, we have
\begin{equation}\label{5.8}
\left.
\begin{aligned}[b]  
&\int\limits_{\Omega\backslash
	\mathcal{W}_{R\epsilon}(x_0)} e^{\lambda_N|\phi_\epsilon|^{\frac{N}{N-1}}}dx\\
\geq& \int\limits_{\Omega\backslash
	\mathcal{W}_{2R\epsilon}(x_0)} \bigg(1+\frac{\lambda_N^{N-1}|\phi_\epsilon|^N}{(N-1)!}\bigg)dx\\  
=& |\Omega|-|	\mathcal{W}_{2R\epsilon}(x_0)|+C^{-\frac{N}{N-1}}\frac{\lambda_N^{N-1}}{(N-1)!}(\|G\|_N^N+O((R\epsilon)^N(-\log(R\epsilon))^N)\\
=& |\Omega|+C^{-\frac{N}{N-1}}\frac{\lambda_N^{N-1}}{(N-1)!}\|G\|_N^N+O((-\log \epsilon)^{-\frac{N}{N-1}}).
\end{aligned}
\right.   
\end{equation}
Thus  
\begin{equation}\label{4.17}
\left.
\begin{aligned}[b]  
\int_{\Omega} e^{\phi_\epsilon}dx=&\int_{\mathcal{W}_{R\epsilon}(x_0)} e^{\lambda_N|\phi_\epsilon|^{\frac{N}{N-1}}}dx+\int\limits_{\Omega\backslash
	\mathcal{W}_{R\epsilon}(x_0)} e^{\lambda_N|\phi_\epsilon|^{\frac{N}{N-1}}}dx\\
   \geq&|\Omega|+\kappa_Ne^{\lambda_N A_{x_0}+\sum_{k=1}^{N}\frac{1}{k}}+C^{-\frac{N}{N-1}}\frac{\lambda_N^{N-1}}{(N-1)!}\|G\|_N^N+O((-\log \epsilon)^{-\frac{N}{N-1}}),
\end{aligned}
\right.   
\end{equation}
By choosing a small $\epsilon > 0$, we conclude $\int_{\Omega}e^{\phi_\epsilon}dx>|\Omega|+\kappa_Ne^{\lambda_N A_{x_0}+\sum_{k=1}^{N}\frac{1}{k}}.$ Hence we finish the proof of Theorem \ref{thm1.2}.
\end{proof}

\def\refname{References }

\end{document}